\documentclass{article}
\RequirePackage{plain,amssymb, amsfonts, amsthm,amsmath}
\RequirePackage{hyperref}

\usepackage[active]{srcltx}

\theoremstyle{plain}
\newtheorem{theorem}{Theorem}[section]
\newtheorem{prop}[theorem]{Proposition}
\newtheorem{lemma}[theorem]{Lemma}

\newtheorem{example}{Example}[section]
\newtheorem{assumption}{Assumption}

\long\def\comment#1{{}}

\def\vf#1{\boldsymbol{#1}}

\def\Grp#1{\left(#1\right)}
\def\Cbr#1{\left\{#1\right\}}
\def\Sbr#1{\left[#1\right]}
\def\Abs#1{\left|#1\right|}

\def\normx#1{\|#1\|}

\def\Ang#1{\left\langle#1\right\rangle}

\def\nth#1{\frac{1}{#1}}
\def\cf#1{\mathbf{1}\Cbr{#1}}
\def\tp{\sp{\top}}
\def\eno#1#2{{#1}_1, \ldots, {#1}_{#2}}
\def\gv{\,|\,}

\def\argf#1{\mathop{\arg#1\,}}

\def\sm #1#2{\sum_{#1\le #2}}
\def\mx #1#2{\max_{#1\le #2}}
\def\Sp#1{\sp{(#1)}}

\def\mean{\text{\sf E}}
\def\var{\text{\sf Var}}
\def\prob{\text{\sf Pr}}
\def\Reals{\mathbb{R}}

\def\dev#1#2{\Sbr{\hspace{#1}\Sbr{#2}\hspace{#1}}}
\def\devx#1#2{[\hspace{#1}[{#2}]\hspace{#1}]}
\def\est#1{\smash{\widehat{#1}}}
\def\sppt{{\rm spt}}
\def\dd{\mathrm{d}}

\def\sign{\text{sign}}

\def\cY{\mathcal{Y}}

\def\rx{\varepsilon}
\def\vtheta{{\vf\theta}}
\begin{document}

\begin{center}
  \Large{\bf
    Stochastic Lipschitz continuity for high dimensional Lasso with
    multiple linear covariate structures or hidden linear covariates
  }
  \\[1em]
  \normalsize
  Short title: Estimation for various high dimensional linear structures
  \\[1em] 
  Zhiyi Chi \\
  Department of Statistics\\
  University of Connecticut \\
  215 Glenbrook Road, U-4120 \\
  Storrs, CT 06269, USA \\[1em]
  \today
\end{center}

\begin{abstract}
  Two extensions of generalized linear models are considered.  In the
  first one, response variables depend on multiple linear combinations
  of covariates.  In the second one, only response variables are
  observed while the linear covariates are missing.  We derive
  stochastic Lipschitz continuity results for the loss functions
  involved in the regression problems and apply them to get bounds on
  estimation error for Lasso.  Multivariate comparison results on
  Rademacher complexity are obtained as tools to establish the
  stochastic Lipschitz continuity results.

  \medbreak\noindent

  \medbreak\noindent

  \medbreak\noindent
  \textit{Acknowledgement.\/} Research partially supported by NSF
  grant DMS-07-06048.

\end{abstract}
\section{Introduction}

In recent years, much attention has been paid to regularized
regression for high dimensional linear models \cite{bickel:09:as,
bunea:etal:07, bunea:10:as, candes:plan:09, ravikumar:etal:10, zhang:09,
zhao:yu:06}.  Meanwhile, a much smaller body of works has
been devoted to such regression for high dimensional
generalized linear models 
\cite{vandegeer:08}.  Despite the impressive progress, the full
potential of regularized regression for models with underlying linear
structures seems far from being fully explored.

Regression is a type of optimization.  In current literature on high
dimensional generalized linear models, the target, or loss, functions
being optimized have the form $\gamma_i(X_i\tp \vf u, Y_i)$, where
$\gamma_i$ is a known function, $X_i$ a high dimensional covariate,
$\vf u$ a parameter, and $Y_i$ a response variable.  Two possible
extensions of the regression can be identified as follows.  First,
instead of one parameter vector, a small number of parameter vectors
may appear in a model, so that the loss functions become
$\gamma_i(X_i\tp \vf u_1, \ldots, X_i\tp\vf u_k, Y_i)$. 
Parameter estimation involving multiple linear combinations of
covariates has been considered at least in neuroscience, where
multiple informative dimensions of visual signals of very high
dimension need to be estimated based on a relatively small amount of
data, so that the neural activity in a visual system may be better
characterized \cite{atencio:08}.  The goal of the effort is rather
ambitious, which is to estimate the functional form of $\gamma_i$
nonparametrically along with a few informative dimensions
characterized by $\eno{\vf u} k$.  However, it seems that a rigorous
development toward this goal is difficult using currently available
statistical methods.  A more modest goal is to estimate $\eno{\vf u}
k$ while having $\gamma_i$ fixed.  For apparently more flexible loss
functions $\gamma_i(\theta, X_i\tp\vf u_1, \ldots,
X_i\tp\vf u_k, Y_i)$, where $\theta$ is a parameter controlling the
shape of $\gamma_i$, by adding auxiliary covariates into $X_i$, one
can reformulate them into $\gamma_i(Z_i\tp \vf v_1, \ldots, Z_i\tp\vf
v_l, Y_i)$.  Of course, the dimension of $\theta$ has to be low.  Once
the dimension of $\theta$ gets high, the estimation becomes no less
challenging than the aforementioned nonparametric estimation.

Second, instead of both the covariates and response variables being
observed, $X_i$ may be missing and only $Y_i$ are observed.  To be
specific, suppose we wish to use regularized likelihood estimation.
Each loss function is then the logarithm of the marginal of $Y_i$ at
parameter value $\vf u$, which no longer has the form $\gamma_i(X_i\tp
\vf u, Y_i)$.  Parameter estimation with missing data is certainly of
interest in its own right.  Naturally, one has to make more
assumptions on the structure of the random object $(X_i, Y_i)$ in
order to estimate $\vf u$.  The issue is, provided such assumptions
are made, whether regularized regression can still work when $\vf u$ 
is of high dimension.

To attack these two regression problems, we use a method in
\cite{chi:10b}, which establishes estimation precision by first
obtaining certain stochastic Lipschitz continuity results for the
total loss function and then combining it with $\ell_1$ regularized
regression (Lasso).  Basically, if 
$L(\vf u)$ denotes the empirical total loss, with $\vf u = (\eno{\vf u}
k)$, then the so called stochastic Lipschitz continuity is concerned
with the upper tail behavior of the supremum of 
$$
\frac{|(L(\vf u) - \mean L(\vf u))- (L(\vf v)-\mean L(\vf v))|}
{\sm j k \normx{\vf u_k - \vf v_k}_1}, \quad\vf u\not=\vf v,
$$
where $\eno{\vf u} k$ are allowed to vary over a certain domain of
parameter values.  Note that we are interested in the fluctuation of
the loss, i.e., $L(\vf u) - \mean L(\vf u)$, rather than the loss
itself.  If $\eno{\vf v} k$ are also allowed to vary over
the domain, then by definition, the supremum is just the Lipschitz
coefficient of $L(\vf u) - \mean L(\vf u)$.  With a little abuse of
language, if $\eno{\vf v} k$ are fixed, the supremum will be referred
to as the local Lipschitz coefficient at $\vf v$.  As in the display,
we shall always consider stochastic Lipschitz continuity with respect
to (wrt) $\ell_1$ norm.

For linear models, stochastic Lipschitz continuity has already been
recognized as a useful tool to study high dimensional Lasso (cf.\
\cite{bickel:09:as, bunea:07:as} and references therein).  The issue
becomes significantly more involved for the problems we consider.  Our
solution requires certain comparison results on Rademacher complexity
\cite{ledoux:91}.  The topic of Rademacher complexity has recently
generated quite amount of interest \cite{ambroladze:07, bartlett:05,
 koltchinskii:06, meir:zhang:04, zhang:yu:05}.  The results in these
works are on processes of the type $\sum \rx_i f_i(t_i)$, where
$\rx_i$ are independent Rademacher variables, i.e., $\Pr\Cbr{\rx_i =
  1} = \Pr\Cbr{\rx_i=-1}=1/2$.  Without going into detail, the point
is that $f_i$ are univariate, i.e., $t_i\in \Reals$.  It turns out we
need comparison results involving multivariate $f_i$.  However, it
appears that such results are not yet available in the literature.  As
a technical preparation, two such results will be given in Section
\ref{sec:compare}, both having the classical form as in
\cite{ledoux:91}.  Similar to \cite{ambroladze:07}, some of the
results can be extended to symmetric integrable $\rx_i$ that need not
be identically distributed.  This is potentially useful for dealing
with stochastic Lipschitz continuity involving unbounded noise terms
\cite{chi:10b}, such as subgaussian ones that allow similar measure
concentration as bounded noise (cf.\ \cite{ledoux:01}, p.~41).  A
detailed study on this, however, is beyond the scope of the article.

Sections \ref{sec:stoch-lip} and \ref{sec:reg} deal with regression
involving multiple linear combinations of covariates.  First, in
Section \ref{sec:stoch-lip}, we use the multivariate comparison
results in Section \ref{sec:compare} to derive stochastic Lipschitz
continuity for loss functions of the form $\sm i N \gamma_i(Z_i\vf u,
Y_i)$, where $\eno Z N$ are fixed matrices and $\eno Y N$ are
independent random variables.  It is not hard to see that such loss
functions include $\sm i N \gamma_i(X_i\tp\vf u_1, \ldots, X_i\tp\vf
u_k, Y_i)$ as special cases, if each $Z_i$ is appropriately
constructed from $X_i$.  For the parameter estimation considered in
Section \ref{sec:reg}, local stochastic Lipschitz continuity is
sufficient for our need.  However, (the usual) stochastic Lipschitz
continuity is known in the context of linear regression and not
difficult to be established following the method for local stochastic
Lipschitz continuity.  For completeness, we shall give a result on
stochastic Lipschitz continuity as well.  In Section \ref{sec:reg}, we
apply the results in Section \ref{sec:stoch-lip} to the Lasso
estimator
$$
(\eno{\est\vtheta} k)
= \argf\min_{\vf u \in D} \Cbr{
  \sm i N \gamma_i(X_i\tp\vf u_1, \ldots, X_i\tp\vf u_k, Y_i) +
  \lambda \sm j k \normx{\vf u_j}_1},
$$
where $D$ is a domain of parameter values and $\lambda>0$ is a tuning
parameter.  Comparing to the case where $k=1$, the issue is that the
Lasso only gives a comparison between $\sm jk\normx{\est\vtheta_j}_1$
and $\sm j k \normx{\vtheta_j}_1$, where $\vtheta_j$ are the true
parameter values.  However, it does not provide direct comparisons
between $\normx{\est\vtheta_j}_1$ and $\normx{\vtheta_j}_1$ for
individual $j\le k$, which are needed to bound the total $\ell_2$
error of $\est\vtheta_j$.  This issue can be resolved by using an
eigenvalue condition on the design matrix consisting of $X_i$
\cite{bickel:09:as}.

Section \ref{sec:hidden} deals with $\ell_1$ regularized likelihood
estimation when the covariates are missing.  Most effort of this
Section is spent on establishing stochastic Lipschitz continuity for
loss functions expressed, roughly speaking, as $\sm i N \ln\int
f(x\tp\vf u, Y_i)\, d\mu(x\gv Y_i)$.  As can be expected, the
logarithmic and integral transformations in the expression are the
major obstacles to the exploitation of the implicit linearity.
Comparison results on Rademacher complexity, including those in
Section \ref{sec:compare}, will be invoked to get them out of the
way.  After stochastic Lipschitz continuity is in place, the rest of
the work is similar to the full data case and actually requires fewer
technical assumptions.  Finally, proofs of auxiliary results are
collected in the Appendix.

\subsection{Notation}
For $\vf s\in \Reals^k$, denote by $\eno s k$ its coordinates
and $\sppt(\vf s)$ its support, i.e, $\{j: s_j\not=0\}$.  For
$J\subset\{1,\ldots, k\}$, denote by $\pi_J$ the function that maps
$\vf s$ to $(s_1', \ldots, s_k')$ with $s_j'=s_j\cf{j\in J}$.
For $q\in [1,\infty]$, the $\ell_q$ norm of $\vf s\in \Reals^k$ is
$$
\normx{\vf s}_q = 
\begin{cases}
  \Grp{\sm j k |s_j|^q}^{1/q} & \text{if } q<\infty,
  \\[2ex]
  \mx j k |s_j| & \text{if } q=\infty.
\end{cases}\qquad
$$

A function $h$ from $\Reals^k$ to $\Reals$ is called $(M,
\ell_q)$-Lipschitz, if
$$
|h(\vf t) - h(\vf s)| \le M\normx{\vf t-\vf s}_q, \qquad
\text{for all}\ \ \vf s,\, \vf t\in\Reals^k.
$$
If $h$ is defined on $\Reals$, then, as all the $\ell_q$ norms are the
same on $\Reals$, we simply say $h$ is $M$-Lipschitz.

For any random variable $\xi$, denote
$$
\dev{-.35ex}{\xi} = \xi - \mean \xi.
$$
If $\eno X N$ and $\eno Y N$ are independent random variables, denote
by $\mean_X$ (resp.\ $\mean_Y$) the integral wrt the (marginal) law of
$\eno X N$ (resp.\ $\eno Y N$).

With a little abuse of notation, by $x=\argf\min f$ we mean $f(x)=\min
f$ and that the minimizer of $f$ may not be unique.  The same
interpretation applies to $\argf\max$, $\argf\sup$ and $\argf\inf$.

\section{Comparison theorems for multivariate functions}
\label{sec:compare}
Let $N\ge 1$ and $k\ge 1$ be integers.  Denote $V=\Reals^k$ and denote
elements in $V^N$ by $\vf t=(\eno{\vf t} N)$, with $\vf t_i = (t_{i1},
\ldots, t_{ik})\in V$.   In this section, $\rx_h$ and $\rx_{ij}$ will
always denote Rademacher variables.  Furthermore, they are always
assumed to be independent from each other.

\begin{theorem} \label{thm:compare0}
  Let $T\subset V^N$ be a bounded set and $\eno h N$ be functions $V\to
  \Reals$ such that each $h_i$ is $(M_i, \ell_\infty)$-Lipschitz and
  satisfies the vanishing condition that $h_i(\vf t)=0$ if some
  $t_j=0$.  Then, for any function  $\Phi: [0,\infty)\to \Reals$
  convex and nondecreasing,
  \begin{align} \label{eq:compare1}
    \mean\Phi\Grp{\nth 2\sup_{\vf t\in T} \Abs{\sm iN \rx_i h_i(\vf
        t_i)}}
    \le
    \mean\Phi\Grp{\Grp{\sup_{\vf t\in T} \sum_{i,j}
        M_i \rx_{ij} t_{ij}}^+}
    \le
    \mean\Phi\Grp{\sup_{\vf t\in T} \Abs{\sum_{i,j} M_i \rx_{ij}
        t_{ij}}}.
  \end{align}
  Furthermore, for $G: \Reals\to \Reals$ convex and nondecreasing,
  \begin{align} \label{eq:compare2}
    \mean G\Grp{\sup_{\vf t\in T} \sm iN \rx_i h_i(\vf t_i)}
    \le
    \mean G\Grp{\sup_{\vf t\in T} \sum_{i,j} M_i \rx_{ij} t_{ij}}.
  \end{align}
\end{theorem}

The vanishing condition in Theorem \ref{thm:compare0} is satisfied,
for example, by $t_1 f(t_2)$, where $f$ is Lipschitz with $f(0)=0$.
In general, while a function $h$ with $h(\vf 0)=0$ may
not satisfy the condition, it always allows a
decomposition into a sum of functions each satisfying the condition.
For example, if $h$ is defined on $\Reals^2$, then $h(s,t) 
= f(s,t) + h(0,t)+h(s,0)$ with $f(s,t) = h(s,t) - h(0,t)-h(s,0)$,
$h(0,t)$, $h(s,0)$ each satisfying the vanishing condition.  The
decomposition leads to the following result.

\begin{theorem} \label{thm:compare}
  Let $T\subset V^N$ be a bounded set and $\eno h N$ be functions
  $V\to \Reals$ such that each $h_i$ is $(M_i, \ell_\infty)$-Lipschitz
  with $h_i(\vf 0)=0$.  For $j\le k$, let $T_j = \{(t_{1j}, \ldots,
  t_{Nj}): (\eno{\vf t} N)\in T\}\subset 
  \Reals^N$.  Then
  \begin{align} \label{eq:compare4}
    \mean\sup_{\vf t\in T} \Abs{\sm iN \rx_i h_i(\vf t_i)}
    \le
    \beta_k \sm j k \mean\sup_{\vf s\in T_j} \Abs{\sm iN \rx_i M_i
      s_i},
  \end{align}
  where $\beta_k$ is a universal constant that can be set no greater
  than $3^k+3^{k-1}-2^k$.
\end{theorem}

Similar to the univariate results in \cite{ambroladze:07},  Theorem
\ref{thm:compare} remains true if $\rx_i$ are replaced with
independent integrable symmetric variables $\gamma_i$.  Indeed, by
$(\eno \gamma N)\sim (\rx_1|\gamma_1|, \ldots, \rx_N|\gamma_N|)$,
where $\eno\rx N$ are independent from $\gamma_i$, the result follows
by first integrating over $\rx_i$ while conditioning on $|\gamma_i|$,
and then integrating over $|\gamma_i|$.  As can be seen, $\gamma_i$
need not be identically distributed in the argument.

\subsection{Proofs}
\begin{lemma} \label{lemma:compare}
  Let $h: V\to \Reals$ be $(M, \ell_\infty)$-Lipschitz and satisfies
  the condition that $h(\vf t)=0$ if some $t_j=0$.  Suppose $S\subset
  \Reals\times V$ is bounded.  Then for any $G: \Reals\to \Reals$
  convex and nondecreasing,
  \begin{gather} \label{eq:compare0}
    \mean G\Grp{
      \sup_{(x,\vf s)\in S}(x+\rx_0 h(\vf s))
    }
    \le\mean G\Grp{\sup_{(x,\vf s)\in S}\Grp{x+
        M\sm j k \rx_j s_j}},
  \end{gather}
\end{lemma}

\begin{proof}
  First, we notice that 
  \begin{align} \label{eq:contract1}
    |h(\vf t)| \le M \min(|t_1|, \ldots, |t_k|).
  \end{align}
  Indeed, for any $j\le k$, let $\vf s=\pi_{\{1,\ldots,k\}
    \setminus\{j\}} \vf t$, i.e. $\vf s$ has the same coordinates as
  $\vf t$ except the $j$th one being 0.  Then $h(\vf s)=0$, and as $h$
  is $(M,\ell_\infty)$-Lipschitz, $|h(\vf t)| = |h(\vf t)-h(\vf s)|
  \le M\normx{\vf t-\vf s}_\infty = M|t_j|$.

  We shall assume $S$ is compact.  By dominated convergence, the
  assumption causes no loss of generality.  Also, we shall assume
  $M=1$.  Otherwise, we can use change of variables $\vf s' = M\vf s$,
  $h'(\vf s') = h(\vf s'/M)$ to reduce to this case.  Let
  $$
  (a,\vf u) = \argf\sup_{(x,\vf s)\in S} (x+h(\vf s)), \quad
  (b,\vf v) = \argf\sup_{(x,\vf s)\in S} (x-h(\vf s)).
  $$
  Then
  \begin{align*}
    \mean G\Grp{\sup_{(x,\vf s)\in S}(x+\rx_0 h(\vf s))}
    = \nth 2 \Sbr{G\Grp{a + h(\vf u)} + G\Grp{b - h(\vf v)}}.
  \end{align*}
  Assume $|u_i - v_i| = \normx{\vf u - \vf v}_\infty$.   Then, in
  order to show \eqref{eq:compare0}, it suffices to show
  \begin{align} 
    &
    G\Grp{a+ h(\vf u)} + G\Grp{b - h(\vf v)} \nonumber \\
    &
    \qquad\le
    \begin{cases}
      \displaystyle
      \mean\Sbr{
        G\Grp{a + u_i + \sum_{j\not=i} \rx_j u_j}
        + G\Grp{b - v_i + \sum_{j\not=i} \rx_j v_j}
      }
      & \text{if } u_i\ge v_i, \\[4ex]
      \displaystyle
      \mean\Sbr{
        G\Grp{a - u_i + \sum_{j\not=i} \rx_j u_j}
        + G\Grp{b + v_i + \sum_{j\not=i} \rx_j v_j}
      }
      & \text{else.}
    \end{cases}
    \label{eq:conditional}
  \end{align}
  
  Suppose $u_i\ge v_i$.  Since $G$ is convex, by Jensen's inequality, 
  \eqref{eq:conditional} is implied by
  \begin{align} \label{eq:contract-ineq}
    G\Grp{a+ h(\vf u)} + G\Grp{b - h(\vf v)} 
    \le G\Grp{a+ u_i} + G\Grp{b - v_i}.
  \end{align}

  Following \cite{ledoux:91}, the proof of \eqref{eq:contract-ineq} is
  divided into 3 cases.
  
  1) $u_i\ge v_i\ge 0$.  Now \eqref{eq:contract-ineq} is equivalent
  to $G\Grp{b - h(\vf v)} -G\Grp{b - v_i} \le G\Grp{a+ ru_i} -
  G\Grp{a+ h(\vf u)}$, so by the convexity of $G$, we only need to
  show
  \begin{align} \label{eq:contract-ineq2}
    u_i - h(\vf u) \ge v_i - h(\vf v)\ge 0,
    \quad a + h(\vf u) \ge b-v_i.
  \end{align}
  Since $h$ is $(1,\ell_\infty)$-Lipschitz, $h(\vf u)-h(\vf v) \le
  \normx{\vf u- \vf v}_\infty =u_i-v_i$, and so $u_i - h(\vf u) \ge
  v_i - h(\vf v)$.   By \eqref{eq:contract1}, $v_i-h(\vf
  v)\ge v_i-|h(\vf v)| \ge 0$ and together with the definition of $\vf
  u$, $a + h(\vf u) \ge b + h(\vf v) \ge b -
  v_i$.  Thus \eqref{eq:contract-ineq2} follows.
  
  2) $u_i\ge 0 \ge v_i$.  By \eqref{eq:contract1}, $a + u_i
  \ge a + |h(\vf u)| \ge a + h(\vf u)$, $b - v_i =  b + |v_i| \ge b +
  |h(\vf v)| \ge b - h(\vf v)$.  Since $G$ is nondecreasing, then
  \eqref{eq:contract-ineq} holds.

  3) $0\ge u_i\ge v_i$.  It suffices to show $G\Grp{a + h(\vf u)} -
  G\Grp{a + u_i} \le G\Grp{b - v_i} - G\Grp{b - h(\vf v)}$.  The proof
  is completely similar to case 1).  We thus have shown
  \eqref{eq:conditional} for the case $u_i\ge v_i$.  The proof for
  $u_i\le v_i$ is completely similar.  
\end{proof}

\begin{proof}[Proof of Theorem \ref{thm:compare0}]
  First, \eqref{eq:compare1} is a consequence of \eqref{eq:compare2}.
  To see this, let $H_{\vf t} = (h_1(\vf t_1), \ldots, h_N(\vf t_N))$
  and $\vf\rx = (\eno \rx N)$.   Then
  $$
  \sup_{\vf t\in T} \Abs{\Ang{\vf\rx, H_{\vf t}}}
  = \sup_{\vf t\in T} (\Ang{\vf\rx, H_{\vf t}}^+ + 
  \Ang{\vf\rx, H_{\vf t}}^-)
  \le 
  \sup_{\vf t\in T} \Ang{\vf\rx, H_{\vf t}}^+ +\sup_{\vf t\in T}
  \Ang{\vf\rx, H_{\vf t}}^-. 
  $$
  So by the nondecreasing monotonicity and convexity of $\Phi$,
  \begin{align*}
    \mean\Phi\Grp{\nth 2\sup_{\vf t\in T} \Abs{\Ang{\vf\rx, H_{\vf t}}}}
    \le
    \nth 2\Sbr{
      \mean\Phi\Grp{\sup_{\vf t\in T} \Ang{\vf\rx, H_{\vf t}}^+}
      +\mean\Phi\Grp{\sup_{\vf t\in T} \Ang{\vf\rx, H_{\vf t}}^-}
    }.
  \end{align*}
  Since $\vf\rx \sim -\vf\rx$, then $\sup_{\vf t\in T} \Ang{\vf\rx,
    H_{\vf t}}^- \sim \sup_{\vf t\in T} \Ang{-\vf\rx, H_{\vf t}}^-
  =\sup_{\vf t\in T} \Ang{\vf\rx, H_{\vf t}}^+$, which together with
  the previous inequality yields
  $$
  \mean\Phi\Grp{\nth 2\sup_{\vf t\in T} \Abs{\Ang{\vf\rx, H_{\vf t}}}}
  \le
  \mean\Phi\Grp{\sup_{\vf t\in T} \Ang{\vf\rx, H_{\vf t}}^+}
  =
  \mean\Phi\Grp{\Grp{\sup_{\vf t\in T} \Ang{\vf\rx, H_{\vf t}}}^+},
  $$
  where the equality follows from the fact that $\sup_{a\in A} a^+ =
  (\sup_{a\in A} a)^+$ for any $A\subset \Reals$.
  Now use the fact that $G(x)=\Phi(x^+)$ is convex and increasing and
  \eqref{eq:compare2} to get
  \begin{align*}
    \mean\Phi\Grp{\nth 2\sup_{\vf t\in T}
      \Abs{\Ang{\vf\rx, H_{\vf t}}}
    }
    \le
    \mean\Phi\Grp{\Grp{\sup_{\vf t\in T} \sum_{i,j} M_i \rx_{ij}
        t_{ij}}^+}. 
  \end{align*}
  This proves the first inequality in \eqref{eq:compare1}.  By the
  nondecreasing monotonicity of $\Phi$, the second inequality in
  \eqref{eq:compare1} follows.
  
  It remains to show \eqref{eq:compare2},   If $N=1$, then $T\subset
  V$ and by letting $S = \{(0,\vf t): \vf t\in T\}$,
  \eqref{eq:compare2} follows from Lemma \ref{lemma:compare}.
  Suppose $N\ge 2$.   Given $\eno z{N-1}\in \{-1,1\}$, let
  $$
  S = \Cbr{
    \Grp{\sm j{N-1}z_j h_j(\vf t_j),\ \vf t_N}:\ (\eno {\vf t}
    N)\in T
  } \subset \Reals\times V.
  $$
  Then by Lemma \ref{lemma:compare},
  $$
  \mean G\Grp{\sup_{(x, \vf s)\in S}(x+\rx_N h_N(\vf s))}
  \le\mean G\Grp{\sup_{(x, \vf s)\in S}\Grp{
      x+M_N\sm jk \rx_{Nj} s_j
    }
  }.
  $$
  Since $\rx_i$ and $\rx_{ij}$ are independent, this can be written as
  \begin{align*}
    &
    \mean \Sbr{
      G\Grp{\sup_{\vf t\in T} \sm i N \rx_i h_i(\vf t_i)}
      \,\vline\, \rx_i=z_i, \ i\le N-1
    } \\
    &
    \le\mean\Sbr{
      G\Grp{\sup_{\vf t\in T}\Grp{
          \sm i{N-1} \rx_i h_i(\vf t_i) + M_N\sm jk \rx_{Nj} t_{Nj}
        }
      }
      \,\vline\, \rx_i=z_i, \ i\le N-1
    }.
  \end{align*}
  Integrate over $\eno z{N-1}$ to get
  \begin{align*}
    &
    \mean G\Grp{\sup_{\vf t\in T} \sm i N \rx_i h_i(\vf t_i)}
    \le\mean
    G\Grp{\sup_{\vf t\in T}\Grp{
        \sm i{N-1} \rx_i h_i(\vf t_i) + M_N\sm jk \rx_{Nj} t_{Nj}
      }
    }.
  \end{align*}
  Now apply the same argument to the expectation on the right hand
  side, except that we condition on $\rx_i$, $i<N-1$ and $\rx_{Nj}$,
  $j\le k$.  Then the expectation is no greater than
  \begin{align*}
    \mean G\Grp{\sup_{\vf t\in T}\Grp{
        \sm i{N-2} \rx_i h_i(\vf t_i) + \sum_{i=N-1}^N M_i \sm jk
        \rx_{ij} t_{ij} 
      }
    }.
  \end{align*}
  The proof is then finished by induction.
\end{proof}

\begin{proof}[Proof of Theorem \ref{thm:compare}]
  Write $[k]=\{1,\ldots,k\}$ and $\pi_{-j}$ for
  $\pi_{[k]\setminus\{j\}}$.  Then for $i\le N$ and $\vf t\in V$,
  \begin{align} \label{eq:h-decomp}
    h_i(\vf t) = \sum_{J\subset[k]} f_{iJ}(\vf t),
  \end{align}
  where $f_{iJ}(\vf t) = \sum_{I\subset J} (-1)^{|J|-|I|} h_i(\pi_I\vf
  t)$.  Indeed, the right hand side of \eqref{eq:h-decomp} is
  \begin{align*}
    \sum_{J\subset[k]} \sum_{I\subset J} (-1)^{|J|-|I|} h_i(\pi_I\vf
    t)
    &
    = \sum_{I\subset[k]} h_i(\pi_I\vf t) \sum_{I\subset J\subset[k]}
    (-1)^{|J|-|I|} \\
    &
    = \sum_{I\subset[k]} h_i(\pi_I\vf t) (1-1)^{k-|I|} = h_i(\vf t).
  \end{align*}

  Since $h_i(\pi_I\vf t)$ are $(M_i, \ell_\infty)$-Lipschitz and
  $h_i(\pi_\emptyset \vf t) = h_i(\vf0)=0$, for each $J$ with
  $|J|=s>0$, $f_{iJ}$ is $(c_s M_i, \ell_\infty)$-Lipschitz, where
  $c_s = 2^s-1$.  It is easy to see that 
  $f_{iJ}(\vf t)$ only depends on $t_j$ with $j\in J$.  For each 
  $j\in J$, letting $\vf s=\pi_{-j} \vf t$,
  \begin{align*}
    f_{iJ}(\vf s)
    &= \sum_{j\not\in I\subset J} (-1)^{|J|-|I|} h_i(\pi_I \vf s)
    + 
    \sum_{j\not\in I\subset J } (-1)^{|J|-|\{j\}\cup I|}
    h_i(\pi_{\{j\}\cup I}\vf
    s)
    \\
    &= \sum_{j\not\in I\subset J} (-1)^{|J|-|I|} [h_i(\pi_I \vf s)
    -h_i(\pi_{\{j\}\cup I} \vf s)].
  \end{align*}
  For every $I$ not containing $j$, $\pi_I \vf
  s = \pi_{\{j\}\cup I} \vf s = \pi_I \vf t$.   As a result,
  $f_{iJ}(\pi_{-j} \vf t)=0$.  In orther words, as a function only in
  $(t_j, j\in J)$, $f_{iJ}(\vf t)$ vanishes if $t_j=0$ for some $j$.
  Then by Theorem \ref{thm:compare0}
  \begin{align*}
    \mean\sup_{\vf t\in T} \Abs{\sm iN \rx_i f_{iJ}(\vf t_i)}
    \le
    2c_s \mean\sup_{\vf t\in T}
    \Abs{\sm iN M_i\sum_{j\in J} \rx_{ij}t_{ij}}.
  \end{align*}
  Since all $\rx_i$, $\rx_{ij}$ are i.i.d.,
  $$
  \mean\sup_{\vf t\in T}
  \Abs{\sm iN M_i\sum_{j\in J} \rx_{ij} t_{ij}}
  \le \mean\Grp{\sum_{j\in J} \sup_{\vf t\in T}
    \Abs{\sm iN M_i \rx_{ij} t_{ij}}}
  = \sum_{j\in J} \mean\sup_{\vf s\in T_j}
  \Abs{\sm iN \rx_i M_i s_i}.
  $$

  Combining \eqref{eq:h-decomp} and the above bound,
  \begin{align*}
    \mean\Grp{\sup_{\vf t\in T} \Abs{\sm iN \rx_i h_i(\vf t_i)}}
    &
    \le 
    \sum_{J\subset [k]}
    \mean\Grp{\sup_{\vf t\in T} \Abs{\sm iN \rx_i f_{iJ}(\vf t_i)}}\\
    &\le
    2 \sum_{J\subset [k]} (2^{|J|}-1) \sum_{j\in J} 
    \mean\sup_{\vf s\in T_j} \Abs{\sm iN \rx_i M_i s_i}.
  \end{align*}
  By simple combinatorial calculation, the proof is complete.
\end{proof}

\section{Stochastic Lipschitz conditions} \label{sec:stoch-lip}
Let $(Y_1, Z_1)$, \ldots, $(Y_N, Z_N)$ be independent random vectors,
with $Y_i$ taking values in a measurable space $\cY$ and $Z_i$ being
$k\times p$ matrices.  For
$j\le k$, denote by $Z_{ij}\tp$ the $j$th row vector of $Z_i$ and for
$h\le p$, $Z_{ijh}$ the $(j,h)$th entry of $Z_i$.  That is
$$
Z_i = \begin{pmatrix}
  Z_{i1}\tp \\\vdots\\ Z_{ik}\tp
\end{pmatrix}
= \begin{pmatrix}
  Z_{i11} & Z_{i12} & \ldots & Z_{i1p} \\
  \vdots & \vdots & \ddots & \vdots \\
  Z_{ik1} & Z_{ik2} & \ldots & Z_{ikp} 
\end{pmatrix}
$$

Henceforth, we consider the case where $Z_i$ are fixed.  Let $D\subset
\Reals^p$ be a fixed domain.  Then for $i\le N$ and $\vf u\in D$, $Z_i
\vf u\in \Reals^k$.  Define
\begin{align} \label{eq:D-radius}
  M_Z = \mx i N \mx j k \normx{Z_{ij}}_\infty, \quad
  R_D := \sup_{\vf u,\vf v\in D} \normx{\vf u-\vf v}_1.
\end{align}

Suppose $\eno\gamma N$ are real valued functions on $\Reals^k\times
\cY$.  For $j\le k$, denote by $\partial_j$ the first partial
differentiation wrt $t_j$.  We make the following assumption.
\begin{assumption} \label{a:map}
  For all $i\le N$ and $y\in\cY$, $\gamma_i(\vf t,y)$ is first
  order differentiable in $\vf t$, such that
  \begin{align*}
    F_1
    &:
    = \sup\Cbr{
      \Abs{\partial_j \gamma_i(\vf s,y)}:\,
      \vf s\in\Reals^k, \, y\in \cY,\, j\le k,\, i\le N
    }<\infty, \\
    F_2
    &
    := \sup\Cbr{
      \frac{
        \Abs{\partial_j\gamma_i(\vf s,y) - \partial_j\gamma_i(\vf t,y)}
      }{
        \normx{\vf t-\vf s}_\infty
      }:\,\vf s, \vf t\in\Reals^k,\,
      \vf s\not=\vf t,\, y\in \cY,\, j\le k,\, i\le N
    }<\infty.
  \end{align*}
\end{assumption}

\begin{theorem}[Local stochastic Lipschitz continuity]
  \label{thm:local-lip}
  Let $\vtheta\in D$ be fixed.  Under Assumption \ref{a:map},
  for $\vf u\in D$,
  \begin{align} \label{eq:taylor-loc-lip}
    \sm i N \dev{-.35ex}{\gamma_i(Z_i \vf u, Y_i)}
    = \sm i N \dev{-.35ex}{\gamma_i(Z_i\vtheta, Y_i)}
    + 
    \sum_{i,j} \dev{-.35ex}{\partial_j\gamma_i(Z_i\vtheta, Y_i)}
    Z_{ij}\tp(\vf u-\vf \theta) + \xi(\vf u)\tp(\vf u-\vf \theta),
  \end{align}
  where $\xi(\vf u)\in\Reals^p$ is a process with the property that
  for any $q\in (0,1)$,
  \begin{align} 
    \prob\Biggl\{\sup_{\vf u\in D} \normx{\xi(\vf u)}_\infty
    &
    >
    A\sqrt{\ln(2p)\sm j k \mx h p\sm i N Z_{ijh}^2}
    \nonumber\\
    &\qquad
    + B \sqrt{\ln(p/q)\mx h p\sum_{i,j}Z_{ijh}^2}
    + C \ln(p/q)\Biggr\}
    \le q,  \label{eq:loc-lip-tail}
  \end{align}
  where, letting
  \begin{align} \label{eq:phi-psi-loc}
    \phi = M_Z \min(2F_1, F_2 M_Z R_D), \quad
    \psi = k \beta_k M_Z F_2,
  \end{align}
  with $\beta_k$ a universal constant as in Theorem \ref{thm:compare},
  $A=4\sqrt{2 k} R_D\psi$, $B=\sqrt{2k} \phi$, $C=8k \phi$.

  Furthermore, given $q_0\in (0,1)$, for any $q, q'\in (0,1)$
  satisfying $q+q'=q_0$, w.p.\ at least $1-q_0$,
  \begin{align} 
    &\hspace{-2em}
    \sup_{\vf u \in D\setminus\{\vf\theta\}}
    \nth{\normx{\vf u -\vf\theta}_1}
    \Abs{
      \sm i N \dev{-.35ex}{\gamma_i(Z_i \vf u, Y_i)}
      -\sm i N \dev{-.35ex}{\gamma_i(Z_i\vtheta, Y_i)}
    }
    \nonumber \\
    &\hspace{1em}
    \le \sqrt{2k} F_1 \sqrt{\ln(2p/q') \mx h p \sum_{i,j} Z_{ijh}^2}
    + A\sqrt{\ln(2p)\sm j k \mx h p\sm i N Z_{ijh}^2}
    \nonumber\\
    &\quad\hspace{3em}
    + B \sqrt{\ln(p/q)\mx h p\sum_{i,j} Z_{ijh}^2} + C \ln(p/q).
    \label{eq:loc-lip}
  \end{align}
\end{theorem}

\begin{theorem}[Stochastic Lipschitz continuity]
  \label{thm:global-lip}
  Fix an arbitrary $\vtheta\in D$.  Under Assumption
  \ref{a:map}, for $\vf u$ and $\vf v\in D$,
  \begin{align} \label{eq:taylor-glob-lip}
    \sm i N \dev{-.35ex}{\gamma_i(Z_i \vf u, Y_i)}-
    \sm i N \dev{-.35ex}{\gamma_i(Z_i\vf v, Y_i)}
    =
    \sum_{i,j} \dev{-.35ex}{\partial_j\gamma_i(Z_i\vtheta, Y_i)}
    Z_{ij}\tp(\vf u - \vf v)  + \xi(\vf u, \vf v)\tp (\vf u-\vf v),
  \end{align}
  where $\xi(\vf u, \vf v)\in \Reals^p$ is a process with the
  property that for any $q\in (0,1)$,
  \begin{align} 
    \prob\Bigg\{
    \sup_{\vf u, \vf v\in D} \normx{\xi(\vf u, \vf v)}_\infty 
    &>
    \bar A\sqrt{\ln(2p)\sm jk \mx h p \sm i N Z_{ijh}^2} 
    \nonumber\\
    &\qquad
    + \bar B \sqrt{\ln(p/q)\mx h p \sum_{i,j}
      Z_{ijh}^2} + \bar C \ln(p/q)
    \Bigg\} \le q, \label{eq:glob-lip-tail}
  \end{align}
  where, letting
  \begin{align} \label{eq:phi-psi-glob}
    \bar \phi = 2 M_Z \min(F_1, F_2 M_Z R_D), \quad
    \bar \psi = 2k \beta_{2k} M_Z F_2,
  \end{align}
  with $\beta_{2k}$ the universal constant as in Theorem
  \ref{thm:compare}, $\bar A=4\sqrt{2 k} R_D\bar \psi$, $\bar
  B=\sqrt{2k}\bar \phi$, $\bar C=8k \bar \phi$.

  Furthermore, given $q_0\in (0,1)$, for any $q, q'\in (0,1)$ with
  $q+q'=q_0$, w.p.~at least $1-q_0$,
  \begin{align} 
    &\hspace{-2em}
    \sup_{\vf u\not=\vf v \in D}
    \nth{\normx{\vf u -\vf v}_1}
    \Abs{
      \sm i N \dev{-.35ex}{\gamma_i(Z_i \vf u, Y_i)}
      -\sm i N \dev{-.35ex}{\gamma_i(Z_i\vf v, Y_i)}
    }
    \nonumber\\
    &\hspace{1em}
    \le \sqrt{2k} F_1 \sqrt{\ln(2p/q') \mx h p \sum_{i,j} Z_{ijh}^2}
    + \bar A\sqrt{\ln(2p)\sm j k \mx h p\sm i N Z_{ijh}^2}
    \nonumber\\
    &\hspace{3em}
    + \bar B \sqrt{\ln(p/q)\mx h p\sum_{i,j} Z_{ijh}^2} +
    \bar C \ln(p/q).
    \label{eq:glob-lip}
  \end{align}
\end{theorem}

\subsection{Preliminaries}

We shall repeatedly use several fundamental results in probability.
First, the following lemma is a combination of the measure
concentration results in \cite{klein:05:ap,massart:00:ap} tailored for
our needs.
\begin{lemma} \label{lemma:mc}
  Suppose $f_1(\vf u), \ldots, f_N(\vf u)\in \Reals$ are independent
  stochastic processes indexed by $\vf u\in D$, where $D\subset
  \Reals^p$ is a measurable set, such that w.p.~1, each $f_i$ has a
  continuous path.  Suppose there are $a_i\le b_i$, $i\le N$, such
  that w.p.~1, $a_i\le f_i(\vf u) \le b_i$ for all $i\le N$ and $\vf
  u\in D$.  Let 
  $$
  W = \sup_{\vf u\in D} \Abs{\sm i N f_i(\vf u)}.
  $$
  Then for any $s>0$,
  \begin{align} \label{eq:f-hoeff}
    \prob\Cbr{W\ge \mean W + \sqrt{2s\sm i N (b_i-a_i)^2}}
    \le e^{-s},
  \end{align}
  Furthermore, assume $\mean f_i(\vf u)=0$ for all $i\le N$ and
  $\vf u\in D$.  Let $M>0$ such that w.p.~1, $|f_i(\vf u)|\le M$, for
  all $i\le N$ and $\vf u\in D$, and let $S>0$ such that 
  $\sm i N \var(f_i(\vf u)) \le S^2$ for all $\vf u\in D$.  Then for
  any $s>0$,
  \begin{align} \label{eq:mc}
    \prob\Cbr{W \ge 2 \mean W + S \sqrt{2s} + 4 M s} \le e^{-s}.
  \end{align}
\end{lemma}

Next, we need the following comparison inequality involving univariate
functions (cf.\ \cite{ledoux:91}, Theorem 4.12; \cite{vandegeer:08}).
\begin{lemma} \label{lemma:u-compare}
  Let $D\subset\Reals^p$ be a measurable set and $\eno\gamma N$ be
  continuous functions from $D$ to $\Reals$.  Suppose $\eno f N$ are
  continuous functions $\Reals\to \Reals$ that map 0 to 0 and are all
  $M$-Lipschitz for some $M>0$.  If $\eno\rx N$ are i.i.d.\ Rademacher
  variables, then 
  $$
  \mean \sup_{\vf u\in D} \Abs{\sm i N \rx_i f_i(\gamma_i(\vf u))}
  \le 2M\mean \sup_{\vf u\in D} \Abs{\sm i N \rx_i \gamma_i(\vf u)}.
  $$
\end{lemma}

The continuity assumption in the above two lemmas is used to ensure
measurability and is satisfied in the situations we shall consider.
Inequality \eqref{eq:f-hoeff} is referred to as functional Hoeffding
inequality in \cite{massart:00:ap}.  The proof of Lemma \ref{lemma:mc}
is given in Appendix.  Finally, we shall also repeatedly use the
following inequality (cf. \cite{massart:00}, Lemma 5.2)
\begin{lemma} \label{lemma:massart}
  Let $\eno\rx N$ be i.i.d.\ Rademacher variables and $A\subset
  \Reals^p$ a finite set.  Let $A_1 = \{\vf a, -\vf a: \vf a\in A\}$.
  Then
  $$
  \mean \max_{\vf a\in A} \Abs{\sm i N \rx_i a_i}
  \le \max_{\vf a\in A} \normx{\vf a}_2 \times \sqrt{2 \ln |A_1|}
  \le \max_{\vf a\in A} \normx{\vf a}_2 \times \sqrt{2 \ln (2|A|)}.
  $$
\end{lemma}

\subsection{Proof of local stochastic Lipschitz continuity}
We next prove Theorem \ref{thm:local-lip}.  Denote $\vf c_i = Z_i
\vtheta$ for 
$i\le N$ and $\vf c=(\eno{\vf c} N)$.  For $\vf u\in D$, denote 
$$
\vf t_i = Z_i(\vf u-\vtheta) = Z_i\vf u - \vf c_i\in
\Reals^k
$$
and $\vf t = (\eno{\vf t} N)$.  Note that $\vf t_i$ are functions
only in $\vf u$.  For each $i\le N$, denote
$$
f_i(\cdot) = \gamma_i(\cdot, Y_i).
$$
For $j\le k$, denote by $\bar\pi_j$ the map $(\eno x k)\to (\eno x j,
0, \ldots, 0)$.  It is easy to check that, for every $i\le N$,
\begin{align} 
  \gamma_i(Z_i\vf u, Y_i) - \gamma_i(Z_i\vtheta, Y_i)
  &
  =f_i(\vf c_i+\vf t_i) - f_i(\vf c_i)
  \nonumber\\
  &= \sm j k \Grp{\partial_j f_i(\vf c_i) + \varphi_{ij}(\vf
    t_i)}t_{ij},
  \label{eq:decomp}
\end{align}
where, for $\vf s\in \Reals^k$,
$$
\varphi_{ij}(\vf s)
= 
\begin{cases}
  \displaystyle
  \frac{f_i(\vf c_i+\bar\pi_j \vf s)
    - f_i(\vf c_i+\bar\pi_{j-1}\vf s)}{s_j}- 
  \partial_j f_i(\vf c_i)
  & \text{if}\ s_j\not=0 \\[2ex]
  \partial_j f_i(\vf c_i+\bar\pi_{j-1} \vf s) - \partial_j f_i(\vf
  c_i), & \text{if}\ s_j=0.
\end{cases}
$$
Thus $\varphi_{ij}$ is a function $\Reals^k\to \Reals$.  We need some
basic properties of $\varphi_{ij}$.  Recall that $F_1$ and $F_2$ are
defined in Assumption \ref{a:map}.
\begin{lemma} \label{lemma:phi}
  W.p.~1, for all $i\le N$ and $j\le k$, $\varphi_{ij}(\vf
  0)=0$, $|\varphi_{ij}(\cdot)|\le 2F_1$ uniformly,
  and $\varphi_{ij}$ is $(F_2, \ell_\infty)$-Lipschitz on $\Reals^k$.
  Furthermore, for $\vf u\in D$, $|\varphi_{ij}(Z_i(\vf u-\vf
  \theta))|\le F_2 M_Z R_D$.
\end{lemma}

From the decomposition \eqref{eq:decomp} and $t_{ij} = Z_{ij}\tp
(\vf u-\vf \theta)\in\Reals$, 
\begin{align*}
  \sm i N (\gamma_i(Z_i\vf u, Y_i) - \gamma_i(Z_i\vtheta, Y_i))
  =
  \sm iN\sm jk\Grp{\partial_j f_i(\vf c_i) + \varphi_{ij}(\vf t_i)}
  Z_{ij}\tp(\vf u-\vtheta),
\end{align*}
giving
\begin{align*}
  \sm i N \dev{-.35ex}{
    \gamma_i(Z_i\vf u, Y_i) - \gamma_i(Z_i\vf c, Y_i)
  }
  = \sum_{i,j}\dev{-.35ex}{\partial_j f_i(\vf c_i)} Z_{ij}\tp(\vf u -
  \vtheta)
  + \sum_{i,j} \dev{-.35ex}{\varphi_{ij}(\vf t_i)} Z_{ij}\tp(\vf u -
  \vtheta).
\end{align*}
Recall $\vf t_i = Z_i (\vf u-\vtheta)$.  Define for $i\le N$ and $h\le
p$
\begin{align*}
  \xi_{ih}(\vf u) = \sm j k\dev{-.38ex}{\varphi_{ij}(\vf t_i)}
  Z_{ijh},
  \quad 
  \xi_h(\vf u) = \sm i N \xi_{ih}(\vf u),
  \quad 
  W_h = \sup_{\vf u \in D} \Abs{\sm iN\xi_{ih}(\vf u)}.
\end{align*}
Then, letting $\xi(\vf u) = (\xi_1(\vf u), \ldots, \xi_p(\vf u))$, it
is seen \eqref{eq:taylor-loc-lip} holds and
\begin{align} \label{eq:Wh}
  \normx{\xi(\vf u)}_\infty = \mx h p|\xi_h(\vf u)| \le \mx h p W_h.
\end{align}

Given $h$, consider the upper tail of $W_h$.  For $i\le N$ and $j\le
k$, by Lemma~\ref{lemma:phi}, $\Abs{\varphi_{ij}(\vf t_i) Z_{ijh}} \le
\phi$, where $\phi=M_Z\min(2F_1, F_2 M_Z R_D)$ as in
\eqref{eq:phi-psi-loc}.  Then
\begin{align} \label{eq:uniform-bound}
  \Abs{\xi_{ih}(\vf u)} \le 
  \sm j k |\dev{-.38ex}{\varphi_{ij}(\vf t_i)} Z_{ijh}| 
  \le 2k\phi := M_0.
\end{align}
Given $\vf u\in D$, for each $i\le N$, $\xi_{ih}(\vf u)$ is a function
only in $Y_i$.  Therefore, by independence and $\var(\sm jk \nu_j) \le
k\sm j k\var(\nu_j) \le k \sm j k \mean \nu_j^2$ for any random
variables $\eno \nu k\in\Reals$,
\begin{align}
  \var\Grp{\xi_h(\vf u)}
  &
  = 
  \sm i N \var\Grp{\xi_{ih}(\vf u)}
  \le
  k \sum_{i,j}\mean\Grp{\varphi_{ij}(\vf t_i) Z_{ijh}}^2
  \le k \phi^2 \sum_{i,j}Z_{ijh}^2 \le S_0^2,
  \label{eq:var-bound}
\end{align}
where
\begin{align} \label{eq:S0}
  S_0^2 = k \phi \mx h p \sum_{i,j} Z_{ijh}^2.
\end{align}

From \eqref{eq:uniform-bound}, \eqref{eq:S0} and Lemma \ref{lemma:mc},
it follows that
\begin{align} \label{eq:Wh-tail}
  \Pr\Cbr{W_h > 2\mean W_h + S_0\sqrt{2s} + 4M_0 s}\le e^{-s}.
\end{align}
Since $\mean \xi_{ih}(\vf u)=0$, by symmetrization (cf.\ the comment
after Lemma 6.3 in \cite{ledoux:91}) and a simple dominated
convergence argument
\begin{align*}
  \mean W_h 
  &= \mean \sup_{\vf u\in D}
  \Abs{\sm i N
    \Sbr{\hspace{-1ex}
      \Sbr{\sm j k \varphi_{ij}(\vf t_i) Z_{ijh}}\hspace{-1ex}
    }
  } \\
  &\le
  2 \mean \sup_{\vf u\in D} \Abs{
    \sm i N \rx_i \sm j k \varphi_{ij}(\vf t_i) Z_{ijh}
  }
  =
  2 \mean \sup_{\vf t\in T} \Abs{
    \sm i N \rx_i \sm j k \varphi_{ij}(\vf t_i) Z_{ijh}
  },
\end{align*}
where $T = \Cbr{(\eno{\vf t} N): \vf t_i = Z_i \vf (\vf
  u-\vtheta), i\le N}$ and $\eno \rx N$ are i.i.d.\ Rademacher
variables independent of $\eno Y N$.  Given $\eno Y N$, by Lemma
\ref{lemma:phi}, each
$$
\tilde\varphi_i(\vf s) = \sm j k \varphi_{ij}(\vf s) Z_{ijh}
$$
is $(kM_Z F_2, \ell_\infty)$-Lipschitz mapping $\vf 0$ to 0.  Note for 
$j\le k$, $\{(t_{1j}, \ldots, t_{Nj}): (\eno{\vf t}N)\in T\}
= \{(Z_{ij}\tp (\vf u-\vtheta), \ldots, Z_{ij}\tp (\vf u - \vtheta)):
\vf u\in D\}$.  Then by Theorem \ref{thm:compare} and the independence
between $\eno Y N$ and $\eno\rx N$, letting $\psi = k \beta_k M_Z F_2$
as in \eqref{eq:phi-psi-loc},
\begin{align*}
  &\hspace{-2em}
  \mean \sup_{\vf t\in T} \Abs{
    \sm i N \rx_i \sm j k \varphi_{ij}(\vf t_i) Z_{ijh}
  }
  =
  \mean_Y
  \mean_\rx \sup_{\vf t\in T} \Abs{
    \sm i N \rx_i \tilde\varphi_i(\vf t_i)
  }
  \\
  &\hspace{1em}\le 
  \mean_Y \Grp{
    \psi \sm j k \mean_\rx\sup_{\vf t\in T}
    \Abs{\sm i N t_{ij}}
  } =
  \psi\sm j k \mean_Y \mean_\rx\sup_{\vf u\in D}
  \Abs{\sm i N \rx_i Z_{ij}\tp (\vf u-\vtheta)}\\
  &\hspace{5.1cm}\le 
  R_D \psi \sm j k \mean_Y\mean_\rx \mx h p \Abs{\sm i N \rx_i
    Z_{ijh}},
\end{align*}
where the last inequality is due to
$$
\Abs{\sm i N \rx_i Z_{ij}\tp (\vf u - \vtheta)}
\le \normx{\vf u -\vtheta}_1 \mx h p \Abs{\sm i N \rx_i Z_{ijh}}
\le R_D \mx h p \Abs{\sm i N \rx_i Z_{ijh}}
$$
for $j\le k$ and $\vf u\in D$.  By Lemma \ref{lemma:massart}, for each
$j\le k$,
$$
\mean_\rx \mx h p \Abs{\sm i N \rx_i Z_{ijh}}
\le \sqrt{2\ln(2p)} \sqrt{\mx h p \sm i N Z_{ijh}^2}.
$$
The right hand side is independent of the values of $\eno YN$.  We
thus get
\begin{align}
  \mean W_h
  &
  \le 2 R_D\psi\sm jk \mean_Y
  \mean_\rx \mx h p \Abs{\sm iN \rx_i Z_{ijh}} \nonumber\\
  &
  \le
  2R_D\psi \sqrt{2\ln(2p)}
  \sm j k\sqrt{\mx h p \sm i N Z_{ijh}^2} \nonumber\\
  &
  \le 2 R_D\psi \sqrt{2k\ln(2p)}
  \sqrt{\sm j k\mx h p \sm i N Z_{ijh}^2},
  \label{eq:mean-Wh}
\end{align}
where the last inequality is due to Cauchy-Schwartz inequality.  Then
by \eqref{eq:Wh-tail},
\begin{align*}
  \Pr\Cbr{W_h > 4 R_D\psi \sqrt{2k\ln(2p)}
    \sqrt{\sm j k \mx h p \sm i N Z_{ijh}^2}
    + S_0\sqrt{2s} + 4M_0 s}\le e^{-s},
\end{align*}
with $M_0$ and $S_0$ being defined in \eqref{eq:uniform-bound} and
\eqref{eq:S0}.  Let $s = \ln (p/q)$ in the above inequality and sum
over $h\le p$.  By \eqref{eq:Wh} and union-sum inequality,
\eqref{eq:loc-lip-tail} is proved.

To prove \eqref{eq:loc-lip}, by \eqref{eq:taylor-loc-lip} and the
above discussion,
\begin{align*}
  &\hspace{-1cm}
  \sup_{\vf u \in D\setminus\{\vf\theta\}}
  \nth{\normx{\vf u -\vf\theta}_1}
  \Abs{
    \sm i N \dev{-.35ex}{\gamma_i(Z_i \vf u, Y_i)}
    -\sm i N \dev{-.35ex}{\gamma_i(Z_i\vtheta, Y_i)}
  } \\
  &\hspace{2cm}
  \le
  \mx h p \Abs{\sum_{i,j} \dev{-.35ex}{\partial_j\gamma_i(Z_i\vtheta, Y_i)}
    Z_{ijh}} + \mx h p W_h.
\end{align*}
Because of \eqref{eq:loc-lip-tail}, it is enough to show that
\begin{align} \label{eq:loc-theta}
  \prob\Cbr{
    \mx h p \Abs{\sm i N \sm jk \dev{-.35ex}{\partial_j\gamma_i(Z_i\vtheta,
        Y_i)}Z_{ijh}} \ge
    \sqrt{2k} F_1 \sqrt{\ln(2p/q') \mx h p \sum_{i,j} Z_{ijh}^2}
  } \le q'.
\end{align}
Given $h\le p$, $\sm jk \dev{-.35ex}{\partial_j\gamma_i(Z_i\vtheta,
  Y_i)}Z_{ijh}$, $i\le N$, are independent of each other, each having
mean 0 and falling between
$$
\pm F_1 \sm j k |Z_{ijh}| 
-\sm j k \mean[\partial_j\gamma_i(Z_i\vtheta, Y_i)] Z_{ijh}.
$$
Therefore, by Hoeffding inequality (\cite{pollard:84}, p.~191) for
any $s>0$,
\begin{align*}
  &\hspace{-.5cm}
  \prob\Cbr{
    \Abs{\sm i N \sm jk \dev{-.35ex}{\partial_j\gamma_i(Z_i\vtheta,
        Y_i)}Z_{ijh}} \ge
    s
  }
  \le 2\exp\Cbr{
    -\frac{s^2}{2F_1^2\sm i N (\sm j k |Z_{ijh}|)^2}
  } \\
  &\hspace{2cm}
  \le 2\exp\Cbr{
    -\frac{s^2}{2k F_1^2\sm i N \sm j k Z_{ijh}^2}
  }
  \le 2\exp\Cbr{
    -\frac{s^2}{2k F_1^2\mx h p \sum_{i,j} Z_{ijh}^2}
  }.
\end{align*}
Let $s=\sqrt{2k} F_1 \sqrt{\ln(2p/q') \mx h p \sum_{i,j}
  Z_{ijh}^2}$.  Then by the union-sum inequality, \eqref{eq:loc-theta}
follows.

\subsection{Proof of stochastic Lipschitz continuity}
We next prove Theorem \ref{thm:global-lip}.  Since the proof follows
that for the local continuity, we shall only highlight differences in
the proof.  Denote $\vf c=(\eno{\vf c} N)$, with $\vf c_i = Z_i\vf
\theta$.  For any $\vf u$ and $\vf v\in D$, denote $\vf s = (\eno {\vf
  s} N)$, $\vf t=(\eno {\vf t} N)$, with $\vf s_i = Z_i (\vf u-\vf
c)$, $\vf t_i = Z_i(\vf v-\vf u)$.  Then $\vf c_i$, $\vf s_i$, $\vf
t_i\in \Reals^k$.  It is important to note that unlike $\vtheta$, both
$\vf u$ and $\vf v$ are variables.  Again, denote $f_i(\cdot) =
\gamma_i(\cdot, Y_i)$ and $\bar\pi_j$ the map $(\eno x k)\to (\eno x
j, 0, \ldots, 0)$.  Then it is easy to check 
\begin{align} 
  \gamma_i(Z_i\vf u, Y_i) - \gamma_i(Z_i\vf v, Y_i)
  &
  =f_i(\vf c_i+\vf s_i+\vf t_i) - f_i(\vf c_i+\vf s_i)
  \nonumber\\
  &
  = \sm j k \Grp{\partial_j f_i(\vf c_i) + \varphi_{ij}(\vf s_i, \vf
    t_i)} t_{ij},
  \label{eq:decomp-glob}
\end{align}
where for $\vf s$, $\vf t\in \Reals^k$,
$$
\varphi_{ij}(\vf s, \vf t)
= 
\begin{cases}
  \displaystyle
  \frac{f_i(\vf c_i+\vf s+\bar\pi_j\vf t)
    - f_i(\vf c_i+\vf s+\bar\pi_{j-1}\vf t)
  } {t_j} - \partial_j f_i(\vf c_i)
  & \text{if}\ t_j\not=0 \\[2ex]
  \partial_j f_i(c_i+\vf s+\bar\pi_{j-1}\vf t)
  - \partial_j f_i(\vf c_i),
  & \text{if}\ t_j=0.
\end{cases}
$$
The function $\varphi_{ij}(\vf s, \vf t)$ has $2k$ real valued
variates, $\eno s k$, $\eno t k$.
\begin{lemma} \label{lemma:phi-glob}
  For all $i\le N$ and $j\le k$, $\varphi_{ij}(\vf0,\vf0)=0$,
  $|\varphi_{ij}(\cdot, \cdot)|\le 2F_1$ uniformly, and $\varphi_{ij}$
  is $(2F_2, \ell_\infty)$-Lipschitz.  Furthermore, for $\vf u$
  and $\vf v\in D$, $|\varphi_{ij}(Z_i(\vf u -\vtheta), Z_i(\vf 
  v - \vf u))|\le 2 F_2 M_Z R_D$.
\end{lemma}
From decomposition \eqref{eq:decomp-glob}, it follows that
\begin{align} \label{eq:dev-glob}
  \sm i N \dev{-.35ex}{
    \gamma_i(Z_i\vf v, Y_i) - \gamma_i(Z_i\vf u, Y_i)
  }
  =
  \sum_{i,j}\dev{-.35ex}{\partial_j f_i(\vf c_i)}
  Z_{ij}\tp(\vf v-\vf u)
  +
  \sum_{i,j}\dev{-.35ex}{\varphi_{ij}(\vf s_i, \vf t_i)}
  Z_{ij}\tp(\vf v-\vf
  u).
\end{align}
Define for $i\le N$ and $h\le p$
\begin{gather*}
  \xi_{ih}(\vf u, \vf v) = \sm j k
  \dev{-.35ex}{\varphi_{ij}(\vf s_i, \vf t_i)} Z_{ijh},\quad
  \xi_h(\vf u, \vf v) = \sm i N \xi_{ih}(\vf u, \vf v), \quad
  W_h = \sup_{\vf u, \vf v\in D} \Abs{\sm iN \xi_{ih}(\vf u, \vf v)}.
\end{gather*}
Then, letting $\xi(\vf u, \vf v) = (\xi_1(\vf u, \vf v), \ldots,
\xi_p(\vf u, \vf v))$, it is seen \eqref{eq:taylor-glob-lip} holds and
\begin{align} \label{eq:Wh-glob}
  \normx{\xi(\vf u, \vf v)}_\infty = \mx h p|\xi_h(\vf u, \vf v)| \le
  \mx h p W_h.
\end{align}

Fix $h$.  For $i\le N$ and $j\le k$, by Lemma \ref{lemma:phi-glob},
letting $\bar\phi = 2 M_Z \min(F_1, F_2M_Z R_D)$ as in
\eqref{eq:phi-psi-glob},
$$
|\varphi_{ij}(\vf s_i, \vf t_i) Z_{ijh}|\le \bar\phi.
$$
Define $\bar M_0$ and $\bar S_0$ in a similar way as
\eqref{eq:uniform-bound} and \eqref{eq:S0}, except that they are in 
terms of $\bar\phi$ instead of $\phi$.  Then, as in
\eqref{eq:Wh-tail},
\begin{align} \label{eq:Wh-tail-glob}
  \Pr\Cbr{W_h > 2\mean W_h + \bar S_0\sqrt{2s} + 4
    \bar M_0 s}\le e^{-s}.
\end{align}
Notice that given $\eno Y N$, each
$$
\tilde\varphi_i(\vf s, \vf t) = \sm j k \varphi_{ij}(\vf s, \vf t)
Z_{ijh}
$$
is $(2kM_Z F_2, \ell_\infty)$-Lipschitz on $\Reals^k\times \Reals^k$
mapping $(\vf0, \vf0)$ to 0.  Then, following the derivation of
\eqref{eq:mean-Wh},
\begin{align} \label{eq:mean-Wh-glob}
  \mean W_h \le 2 R_D\bar\psi \sqrt{2k\ln(2p)}
  \sqrt{\sm j k\mx h p \sm i N Z_{ijh}^2},
\end{align}
where $\bar\psi = 2k\beta_{2k} M_Z F_2$ as in
\eqref{eq:phi-psi-glob}.  The proof of \eqref{eq:glob-lip-tail} can
then be finished in a similar way as \eqref{eq:loc-lip-tail}.  The
proof of \eqref{eq:glob-lip} is completely similar to
\eqref{eq:loc-lip}. 

\section{Lasso for multiple linear combinations of covariates}
\label{sec:reg}
Suppose $\eno\gamma N$ are measurable functions from
$\Reals^k\times \cY$ to $\Reals$.  Let $\eno X N\in V:=\Reals^m$ be
fixed covariate vectors and denote by $\eno {\vf u} N\in V$
parameters.  In this section, we specialize to the following
multivariate loss functions
\begin{align} \label{eq:mv-loss}
  \gamma_i(X_i\tp \vf u_1, \ldots, X_i\tp\vf u_k, Y_i) = \gamma_i(Z_i\vf
  u, Y_i), \quad i\le N,
\end{align}
where 
\begin{align} \label{eq:regress-Z}
  Z_i = \begin{pmatrix}
    X_i\tp & & \\ &\ddots & \\ & & X_i\tp
  \end{pmatrix} \in \Reals^{p\times k}, \quad
  \vf u = \begin{pmatrix}
    \vf u_1 \\ \vdots \\ \vf u_k
  \end{pmatrix}\in\Reals^p, \quad\text{with }\ p=km.
\end{align}
We assume that the form of $\gamma_i$ is already known and consider
the estimation of $\eno{\vf u} k$.

Corresponding to the loss functions $\gamma_i$, the total expected
loss is
\begin{align} \label{eq:mean-total-loss}
  L(\vf u)
  =
  \sm i N \mean\gamma_i(Z_i \vf u, Y_i).
\end{align}
Let $D\subset \Reals^p$ be a compact domain.  Suppose
\begin{align} \label{eq:lasso-target}
  \vtheta=(\eno{\vtheta} k) = \argf\min_{\vf u \in D}  L(\vf u).
\end{align}
The Lasso estimator for $\vtheta$ is of the form 
\begin{align} \label{eq:lasso-form}
  \est\vtheta = 
  (\eno{\est\vtheta} k)
  &
  = \argf\min_{\vf u \in D} \Cbr{
    \sm i N \gamma_i(X_i\tp\vf u_1, \ldots, X_i\tp\vf u_k, Y_i) +
    \lambda \sm j k \normx{\vf u_j}_1 
  } \nonumber\\
  &
  =
  \argf\min_{\vf u \in D} \Cbr{
    \sm i N \gamma_i(Z_i\tp \vf u, Y_i) + \lambda \normx{\vf u}_1
  },
\end{align}
where $\lambda>0$ is a tuning parameter and in the expression on the
second line, $\vf u$ is treated as a concatenation of $\eno{\vf u} k$.
We shall assume that the minima in \eqref{eq:lasso-target} and
\eqref{eq:lasso-form} are always obtained.  However, neither has to
have a unique minimizer.

Denote by $X$ the $N\times m$ design matrix with row vectors $X_1\tp$,
\ldots, $X_N\tp$.  For $l\ge 1$, let
\begin{align} \label{eq:local-sr}
  \sigma_{X,l} = \max\Cbr{
    \frac{\normx{X\vf v}_2}{\normx{\vf v}_2}: \vf v\in V,\
    1\le |\sppt(\vf v)|\le l
  }.
\end{align}
To utilize a restricted eigenvalue (RE) condition introduced in
\cite{bickel:09:as}, define, for $s\le m$ and $K>0$, 
\begin{align*}
  \kappa_X(s, K):= \min\Cbr{
    \frac{\normx{X\vf v}_2}{\sqrt{N} \normx{\pi_J \vf v}_2}: 
    \ \vf v\in\Reals^m\setminus\{\vf 0\}, \
    \normx{\pi_{J^c}\vf v}_1 \le K \normx{\pi_J \vf v}_1,\
    1\le |J|\le s
  }.
\end{align*}

\begin{theorem} \label{thm:lasso}
  Assume $S = \mx j k |\sppt(\vtheta_j)|<m/2$.  Let $q\in (0,1)$.
  Suppose the following conditions are satisfied.
  \begin{itemize} 
  \item[1)] \emph{(Restricted eigenvalue)}  For some $K>1$,
    $\kappa:=\kappa_X(2 S, K)>0$.
  \item[2)] \emph{(Quadratic lower bound of expected loss)}  For some
    $C_\gamma>0$ and all $i\le N$ and $\vf u\in D$,
    $\mean\gamma_i(Z_i\vf u, Y_i)  - \mean\gamma_i(Z_i\vtheta, Y_i)
    \ge C_\gamma \normx{Z_i(\vf u -  \vtheta)}_2^2$.
  \item[3)] \emph{(Local Lipschitz)} There is $M_q>0$, such that
    w.p.~at least $1-q$,
    \begin{align*}
      \Abs{
        \sm i N\dev{-.35ex}{\gamma_i(Z_i\vf u, Y_i)-
          \gamma_i(Z_i\vtheta, Y_i)
        }
      }
      \le M_q \normx{\vf u - \vtheta}_1
      = M_q \sm j k \normx{\vf u_j - \vtheta_j}_1,
      \quad\text{all }\ \vf u\in D.
    \end{align*}
  \end{itemize}
  Let
  \begin{align} \label{eq:LN}
    \lambda = \frac{(K+1) M_q}{K-1}, \qquad
    L_N = \frac{2M_q K}{N\kappa^2 C_\gamma(K-1)}.
  \end{align}
  Then, using this $\lambda$ in the Lasso estimator
  \eqref{eq:lasso-form}, w.p.~at least $1-q$, 
  \begin{align} \label{eq:lasso-error}
    \normx{\est\vtheta - \vtheta}_2^2
    \le k L_N^2 S \Sbr{
      2+K^2 +
      \frac{2(1+K^2) (N\kappa^2 + \sigma_{X,S}^2 k)}{N\kappa^2} \cf{k>1}
    }.
  \end{align}
\end{theorem}

Comparing to the case $k=1$, \eqref{eq:lasso-error} has a multiple of
$1+ \sigma_{X,S}^2 k/N\kappa^2$.  The constant $\sigma_{X,S}$ is
related to the so called $S$-restricted isometry constant
\cite{candes:tao:07}.  The ratio of $\sigma_{X,S}$ to $\kappa$ bears
some similarity to the condition number of matrix, despite the
constraints imposed on their definitions.

\begin{example} \label{ex:MLE}\rm
  Let $f(y\gv \vf t)$ be probability densities on $\Reals$
  parameterized by $\vf t\in\Reals^k$.  Suppose that given covariate
  $x\in V = \Reals^m$, a response variable $Y$ has density $f(y\gv x\tp
  \vtheta_1, \ldots, x\tp \vtheta_k)$, with $\eno{\vtheta} k\in V$
  being unknown parameter values.  To estimate $\vtheta$, suppose
  $Y_i$ under fixed covariate values $X_i$, $i\le N$, are observed.
  Denote $Z_i$ as in \eqref{eq:regress-Z}.  If it is known that
  $\vtheta = (\eno{\vtheta} k)$ is in a bounded set $D\subset V^k$,
  then by \eqref{eq:lasso-form}, one type of $\ell_1$-regularized
  likelihood estimator of $\vtheta$ is
  $$
  \est\vtheta = \argf\min_{\vf u\in D}\Cbr{
    -\sm iN \ln f(Y_i\gv Z_i\vf u) + \lambda \normx{\vf u}_1
  },
  $$
  where the tuning parameter $\lambda$ will be selected in a moment.
  It is seen that the loss functions $\eno\gamma N$ in the setup are 
  $\gamma_i(\vf t, y) = -\ln f(y\gv \vf t)$ and for any $\vf u\in D$,
  the total expected loss is
  $$
  L(\vf u) = -\sm i N \mean\ln f(Y_i\gv Z_i\tp \vf u) = 
  \sm i N D(Z_i\tp \vf u, Z_i\tp\vtheta) + L(\vtheta),
  $$
  where for any $\vf s$, $\vf t\in\Reals^k$,
  $$
  D(\vf s, \vf t)
  = \int f(y\gv\vf t) \ln \frac{f(y\gv\vf t)}{f(y\gv\vf s)}\,\dd y
  $$
  is the Kullback-Leibler distance from $f(y\gv \vf s)$ to $f(y\gv \vf
  t)$.   It is well known that $D(\vf s, \vf t)\ge 0$ with equality if
  and only if $f(y\gv\vf t)\equiv f(y\gv\vf s)$.  Therefore, $\vtheta$
  minimizes $L(\vf u)$.   However, for high dimensional $V$ and
  relatively small $N$, $\vtheta$ may not be the unique minimizer.
  
  Suppose that all $\vtheta_j$ satisfy $|\sppt(\vtheta_j)|\le
  \dim(V)/2=m/2$.  To bound $\normx{\est\vtheta - \vtheta}_2$, assume
  $X = (\eno X N)\tp$ satisfies the RE Condition 1) in
  Theorem~\ref{thm:lasso}.  Since $D$ 
  is bounded, the set of $Z_i\vf u$, $i\le  N$, $\vf u\in D$ is in a
  compact domain $A$.  Suppose that for some $C_\gamma>0$,
  \begin{align} \label{eq:KL}
    D(\vf s, \vf t) \ge C_\gamma\normx{\vf s-\vf t}_2^2, \quad \vf s, \vf
    t\in A.
  \end{align}
  The above condition is satisfied under mild conditions on the
  regularity of $f(y\gv \vf t)$, using the fact that for fixed $\vf
  t$, the Hessian of $D(\vf s, \vf t)$ at $\vf s=\vf t$ is the Fisher
  information at $\vf t$, which is nonnegative definite.  Then for any
  $i\le N$ and $\vf u\in D$,
  $$
  \mean\gamma_i(Z_i\tp\vf u, Y_i) -
  \mean\gamma_i(Z_i\tp\vtheta, Y_i)
  = 
  D(Z_i\vtheta, Z_i\vf u) \ge C_\gamma\normx{Z_i(\vtheta - \vf u)}_2^2,
  $$
  so Condition 2) in Theorem \ref{thm:lasso} is satisfied.  Finally,
  by Theorem \ref{thm:local-lip}, if $-\ln f(y\gv \vf t)$ are first
  order differentiable in $\vf t$, such that the partial
  derivatives are uniformly bounded and have uniformly bounded
  Lipschitz coefficient, then for any $q\in (0,1)$, there is $M_q$
  such that Condition 3) in Theorem \ref{thm:lasso} is satisfied.  As
  a result, by setting $\lambda$ as in \eqref{eq:LN}, we get a bound
  for $\normx{\est\vtheta-\vtheta}_2$ using \eqref{eq:lasso-error}.
  \qed
\end{example}

\subsection{Proof of Theorem \ref{thm:lasso}}
The proof is divided into 3 steps.
\paragraph{\it Step 1.\/}  The argument in this step has now become
standard \cite{bickel:09:as}.  Let $c=(K-1)/2$, where $K$ is as in
Condition 1).  Then $\lambda = (1+1/c)M_q$.  From the definition of
$\est\vtheta$,
\begin{align} \label{eq:est-ineq}
  L(\est\vtheta) - L(\vtheta)
  \le
  \sm i N \devx{-.35ex}{\gamma_i(Z_i\vtheta, Y_i)}-
  \sm i N\devx{-.35ex}{\gamma_i(Z_i\est\vtheta, Y_i)}  + 
  (1+1/c)M_q(\normx{\vtheta}_1 - \normx{\est\vtheta}_1).
\end{align}
By Condition 2) in Theorem \ref{thm:lasso},
\begin{align*}
  L(\est\vtheta) - L(\vtheta)
  \ge C_\gamma \sm i N \normx{Z_i(\est\vtheta - \vtheta)}_2^2
  &
  = C_\gamma \sm iN \sm jk |X_i\tp(\est\vtheta_j - \vtheta_j)|^2
  \\
  &
  = C_\gamma \sm jk \normx{X(\est\vtheta_j - \vtheta_j)}_2^2.
\end{align*}
Then by Condition 3) and \eqref{eq:est-ineq}, w.p.~at least $1-q$,
\begin{align*}
  C_\gamma \sm j k \normx{X(\est\vtheta_j - \vtheta_j)}_2^2
  &
  \le  M_q \sm j k \normx{\est\vtheta_j - \vtheta_j}_1 +
  (1+1/c)M_q 
  \sm j k (\normx{\vtheta_j}_1 -\normx{\est\vtheta_j}_1).
\end{align*}
Let $\eno J k\subset\{1,\ldots, m\}$ be any sets with $J_j\supset
\sppt(\vtheta_j)$.  Then for each $j\le k$,
\begin{align*}
  &
  \normx{\est\vtheta_j - \vtheta_j}_1 -
  (1+1/c) (\normx{\vtheta_j}_1 - \normx{\est\vtheta_j}_1) 
  \\
  &\qquad
  =
  \normx{\pi_{J_j}\est\vtheta_j - \vtheta_j}_1 + 
  \normx{\pi_{J_j^c}\est\vtheta_j}_1 + 
  (1+1/c) (\normx{\vtheta_j}_1 -
  \normx{\pi_{J_j}\est\vtheta_j}_1
  - \normx{\pi_{J_j^c}\est\vtheta_j}_1) \\
  &
  \qquad\le
  (2+1/c)\normx{\pi_{J_j}\est\vtheta_j - \vtheta_j}_1
  - (1/c)\normx{\pi_{J_j^c}\est\vtheta_j}_1\\
  &
  \qquad=
  (K/c)\normx{\pi_{J_j}\est\vtheta_j - \vtheta_j}_1
  - (1/c)\normx{\pi_{J_j^c}\est\vtheta_j}_1.
\end{align*}
It follows that w.p.~at least $1-q$,
\begin{align} \label{eq:lasso}
  \sm jk \normx{X(\est\vtheta_j - \vtheta_j)}_2^2
  \le \frac{M_q }{C_\gamma c}
  \sm j k (K\normx{\pi_{J_j}\est\vtheta_j - \vtheta_j}_1 
  - \normx{\pi_{J_j^c} \est\vtheta_j}_1).
\end{align}

Fix an instance of $(\eno Y N)$ such that \eqref{eq:lasso} holds.  Let
$\eno A k\subset \{1,\ldots, m\}$ be sets such that
$\sppt(\vtheta_j)\subset A_j$ and $|A_j|=S$.  Then \eqref{eq:lasso}
holds with $J_j=A_j$.  Let
$$
I = \{j\le k: K\normx{\pi_{A_j}\est\vtheta_j - \vtheta_j}_1  \ge
\normx{\pi_{A_j^c}\est\vtheta_j}_1\}.
$$
Then $I\not=\emptyset$.  We shall consider $j\in I$ and $j\not\in I$
separately.

Before moving to the next step, for each $j\le k$, let $B_j$ be the
union of $A_j$ and the indices of the $S$ largest $|\est\theta_{jh}|$
outside of $A_j$.  Then \eqref{eq:lasso} holds with $J_j = B_j$.  It
is now well-known that \cite{candes:tao:07}
\begin{align} \label{eq:ct-ineq}
  \normx{\pi_{B_j^c} \est\vtheta_j}_2^2 
  \le \frac{\normx{\pi_{A_j^c} \est\vtheta_j}_1^2}{S}.
\end{align}
It is easy to see that for $j\le k$, $\normx{\pi_{A_j}
\est\vtheta_j - \vtheta}_1 \le\normx{\pi_{B_j}
\est\vtheta_j - \vtheta}_1$ and $\normx{\pi_{A_j^c} \est\vtheta_j}_1
\ge \normx{\pi_{B_j^c} \est\vtheta_j}_1$.

\paragraph{\it Step 2.\/} 
From \eqref{eq:lasso},
\begin{align*} 
  \sum_{j\in I} \normx{X(\est\vtheta_j - \vtheta_j)}_2^2 
  &\le \frac{M_q}{C_\gamma c} \sum_{j\in I} 
  (K\normx{\pi_{A_j}\est\vtheta_j - \vtheta_j}_1 
  - \normx{\pi_{A_j^c} \est\vtheta_j}_1) \\
  &\le \frac{M_q}{C_\gamma c} \sum_{j\in I} 
  (K\normx{\pi_{B_j}\est\vtheta_j - \vtheta_j}_1 
  - \normx{\pi_{B_j^c} \est\vtheta_j}_1).
\end{align*}
For each $j\in I$, $K\normx{\pi_{B_j}\est\vtheta_j - \vtheta_j}_1
-\normx{\pi_{B_j^c}\est\vtheta_j}_1 \ge K\normx{\pi_{A_j}\est\vtheta_j 
  - \vtheta_j}_1 -\normx{\pi_{A_j^c}\est\vtheta_j}_1 \ge 0$,
so by Condition 1), $N \kappa^2\normx{\pi_{J_j}\est\vtheta_j -
  \vtheta_j}_2^2 \le \normx{X(\est\vtheta_j - \vtheta_j)}_2^2$ holds 
for $J_j = A_j$, $B_j$.  Letting $J_j=A_j$, from the above display and
Cauchy-Schwartz inequality,
\begin{align*}
  &
  N\kappa^2\sum_{j\in I} \normx{\pi_{A_j}\est\vtheta_j -
    \theta_j}_2^2 
  \le \frac{M_qK}{C_\gamma c}
  \sum_{j\in I} \normx{\pi_{A_j}\est\vtheta_j-\vtheta_j}_1 \\
  &\hspace{2.5cm}
  \le \frac{M_qK\sqrt{S}}{C_\gamma c}
  \sum_{j\in I} \normx{\pi_{A_j}\est\vtheta_j-\vtheta_j}_2 
  \le \frac{2M_qK\sqrt{S|I|}}{C_\gamma(K-1)}
  \Grp{
    \sum_{j\in I} \normx{\pi_{A_j}\est\vtheta_j-\vtheta_j}_2^2
  }^{1/2},
\end{align*}
giving
\begin{align} \label{eq:norm-IA}
  \sum_{j\in I} \normx{\pi_{A_j}\est\vtheta_j - \vtheta_j}_2^2 
  \le L_N^2 S|I|.
\end{align}
Likewise, letting $J_j=B_j$, it follows that
\begin{align} \label{eq:norm-IB}
  \sum_{j\in I} \normx{\pi_{B_j}\est\vtheta_j - \vtheta_j}_2^2 
  \le 2L_N^2  S|I|,
\end{align}
where the factor 2 is due to $|B_j|=2S$.

By \eqref{eq:ct-ineq}, \eqref{eq:norm-IA} and
Cauchy-Schwartz inequality, it follows that,
\begin{align*}
  \sum_{j\in I} \normx{\pi_{B_j^c}\est\vtheta_j}_2^2
  \le \sum_{j\in I} \frac{\normx{\pi_{A_j^c} \est\vtheta_j}_1^2}{S}
  &\le K^2\sum_{j\in I} 
  \frac{\normx{\pi_{A_j} \est\vtheta_j - \vtheta_j}_1^2}{S} \\
  &\le K^2\sum_{j\in I} 
  \normx{\pi_{A_j} \est\vtheta_j - \vtheta_j}_2^2.
\end{align*}
Combining the inequality with \eqref{eq:norm-IA} and
\eqref{eq:norm-IB},
\begin{align}
  \sum_{j\in I} \normx{\est\vtheta_j - \vtheta_j}_2^2
  &=
  \sum_{j\in I} (\normx{\pi_{B_j} \est\vtheta_j - \vtheta_j}_2^2
  + \normx{\pi_{B_j^c} \est\vtheta_j}_2^2) \nonumber\\
  &\le
  \sum_{j\in I} (\normx{\pi_{B_j} \est\vtheta_j - \vtheta_j}_2^2
  + K^2\normx{\pi_{A_j} \est\vtheta_j - \vtheta_j}_2^2) \nonumber\\
  &\le (2+K^2) L_N^2 S|I|.  \label{eq:norm-I}
\end{align}

\paragraph{\it Step 3.\/}  We next consider $j\not\in I$.  The idea
is to modify each $\est\vtheta_j$ into some $\tilde\vtheta_j$ that can
be dealt with by the argument in Step 2.  For $j\not\in I$,
$K\normx{\pi_{A_j} \est\vtheta_j - \vtheta_j}_1 <
\normx{\pi_{A_j^c}\est\vtheta_j}_1$. 
Then from \eqref{eq:lasso}, we have both
\begin{gather} \label{eq:norm2}
  \sum_{j\not\in I} \normx{X(\est\vtheta_j - \vtheta_j)}_2^2 
  \le \frac{M_q K}{C_\gamma c} \sum_{j\in I} 
  \normx{\pi_{A_j}\est\vtheta_j - \vtheta_j}_1
  =L_N \kappa^2 N\sum_{j\in I} 
  \normx{\pi_{A_j}\est\vtheta_j - \vtheta_j}_1
\end{gather}
and
\begin{gather} \label{eq:norm3}
  0\le \sum_{j\not\in I} 
  (\normx{\pi_{A_j^c} \est\vtheta_j}_1 -
  K\normx{\pi_{A_j}\est\vtheta_j - \vtheta_j}_1) 
  \le K\sum_{j\in I} \normx{\pi_{A_j}\est\vtheta_j - \vtheta_j}_1.
\end{gather}
By Cauchy-Schwartz inequality and \eqref{eq:norm-IA},
\begin{align} \label{eq:norm-I1}
  \sum_{j\in I} \normx{\pi_{A_j}\est\vtheta_j - \vtheta_j}_1
  &\le 
  \sqrt{S}\sum_{j\in I} \normx{\pi_{A_j}\est\vtheta_j - \vtheta_j}_2 
  \nonumber\\
  &\le 
  \sqrt{S |I|}\Grp{
    \sum_{j\in I} \normx{\pi_{A_j}\est\vtheta_j - \vtheta_j}_2^2
  }^{1/2}
  \le L_N S|I|.
\end{align}
Let $\delta_j = \normx{\pi_{A_j^c} \est\vtheta_j}_1-
K\normx{\pi_{A_j}\est\vtheta_j - \vtheta_j}_1$.  Then $\delta_j>0$
for $j\not\in I$ and by \eqref{eq:norm3} and \eqref{eq:norm-I1}, 
\begin{gather} \label{eq:diff-sum}
  \sum_{j\not\in I} \delta_j \le K L_N S|I|.
\end{gather}
For each $j\not\in I$, define
$$
\tilde\vtheta_j
= \est\vtheta_j + \frac{\delta_j}{KS}\sum_{h\in A_j}
\sign(\est\theta_{jh} - \theta_{jh}) \vf e_h,
$$
where $\sign(x)=\cf{x\ge 0}-\cf{x<0}$ and $\vf e_h$ is the $h$th
standard basis vector of $\Reals^m$.  Then for $h\not\in A_j$,
$\tilde\theta_{jh} = \est\theta_{jh}$, while for $h\in A_j$,
$$
|\tilde\theta_{jh} - \theta_{jh}| =
\Abs{
  \est\theta_{jh} - \theta_{jh} +
  \frac{\delta_j}{KS} \sign(\est\theta_{jh} - \theta_{jh})
}
=|\est\theta_{jh} - \theta_{jh}|+\frac{\delta_j}{KS}.
$$
As a result, for $j\not\in I$,
\begin{gather*}
  K\normx{\pi_{A_j}\tilde\vtheta_j - \vtheta_j}_1
  = K\Grp{
    \normx{\pi_{A_j}\est\vtheta_j - \vtheta_j}_1
    +\sum_{h\in A_j} \frac{\delta_j}{KS}
  }
  =\normx{\pi_{A_j^c} \est\vtheta_j}_1
  = \normx{\pi_{A_j^c} \tilde\vtheta_j}_1,
\end{gather*}
and consequently $\normx{\pi_{B_j^c} \tilde\vtheta_j}_1 \le 
K\normx{\pi_{B_j}\tilde\vtheta_j - \vtheta_j}_1$.  Then by Condition
1),
\begin{align} \label{eq:norm-cI}
  \normx{X(\tilde\vtheta_j - \vtheta_j)}_2^2 
  \ge 
  N\kappa^2\normx{\pi_{B_j}\tilde\vtheta_j-\vtheta_j}_2^2
  \ge
  N\kappa^2\normx{\pi_{A_j}\tilde\vtheta_j-\vtheta_j}_2^2.
\end{align}

On the other hand, by the inequality $\normx{\vf s + \vf t}_2^2
\le 2 (\normx{\vf s}_2^2+\normx{\vf t}_2^2)$ for $\vf s$, $\vf t\in
\Reals^N$, and the inequalities in \eqref{eq:norm2} and
\eqref{eq:norm-I1}
\begin{align}
  \sum_{j\not\in I}\normx{X(\tilde\vtheta_j-
    \vtheta_j)}_2^2
  &\le
  2\sum_{j\not\in I} \Grp{\normx{X(\est\vtheta_j-\vtheta_j)}_2^2
    + \normx{X(\tilde\vtheta_j - \est\vtheta_j)}_2^2
  }\nonumber\\
  &
  \le
  2 L_N \kappa^2 N \sum_{j\in I} 
  \normx{\pi_{A_j}\est\vtheta_j - \vtheta_j}_1
  +2\sum_{j\not\in I}\normx{X(\tilde\vtheta_j - \est\vtheta_j)}_2^2
  \nonumber \\
  &
  \le 2L_N^2 \kappa^2 N S|I|+
  2\sum_{j\not\in I}\normx{X(\tilde\vtheta_j - \est\vtheta_j)}_2^2.
  \label{eq:norm-cI2}
\end{align}
Recall the definition of $\sigma_{X,l}$.  Since
$|\sppt(\tilde\vtheta_j - \est\vtheta_j)|\le |A_j| = S$, then
\begin{align*}
  \sum_{j\not\in I}\normx{X(\tilde\vtheta_j - \est\vtheta_j)}_2^2
  &\le
  \sigma_{X,S}^2\sum_{j\not\in I} 
  \normx{\tilde\vtheta_j - \est\vtheta_j}_2^2 \\
  &=
  \sigma_{X,S}^2\sum_{j\not\in I}\sum_{h\in A_j} (\delta_j/KS)^2 
  = \frac{\sigma_{X,S}^2}{K^2 S}
  \sum_{j\not\in I}\delta_j^2
\end{align*}
Then by \eqref{eq:diff-sum}
\begin{align*}
  \sum_{j\not\in I}\normx{X(\tilde\vtheta_j - \est\vtheta_j)}_2^2
  \le 
  \frac{\sigma_{X,S}^2}{K^2 S}
  \Grp{\sum_{j\not\in I}\delta_j}^2
  \le \sigma_{X,S}^2 L_N^2 S |I|^2.
\end{align*}
Plug this inequality into \eqref{eq:norm-cI2} and combine the result
with \eqref{eq:norm-cI} to get
$$
\sum_{j\not\in I}
\normx{\pi_{A_j}\tilde\vtheta_j-\vtheta_j}_2^2
\le
\sum_{j\not\in I}
\normx{\pi_{B_j}\tilde\vtheta_j-\vtheta_j}_2^2
\le \frac{2S L_N^2 |I| (N\kappa^2 + \sigma_{X,S}^2 |I|)}{N\kappa^2}.
$$
Following the derivation of \eqref{eq:norm-I},
$$
\sum_{j\not\in I} \normx{\tilde \vtheta_j - \vtheta_j}_2^2
\le \frac{2S(1+K^2)L_N^2 |I| (N\kappa^2+\sigma_{X,S}^2 |I|)}{N\kappa^2}.
$$
It is easy to see that $\normx{\est \vtheta_j - \vtheta_j}_2 \le
\normx{\tilde\vtheta_j - \vtheta_j}_2$ for $j\not\in I$.  Therefore,
\begin{align} \label{eq:norm-I2}
  \sum_{j\not\in I} \normx{\tilde \vtheta_j - \vtheta_j}_2^2
  \le \frac{2S(1+K^2)L_N^2 |I| (N\kappa^2 + \sigma_{X,S}^2|I|)}{N\kappa^2}.
\end{align}
Note that the left hand is 0 if $k=1$.  Therefore, we can multiply
the right hand side by $\cf{k>1}$.  Finally, combining
\eqref{eq:norm-I} and \eqref{eq:norm-I2}, the proof is complete.

\section{Hidden variable model} \label{sec:hidden}
Suppose $(\omega_1, Y_1)$, \ldots, $(\omega_N, Y_N)$ are independent
random vectors taking values in $\Omega\times\cY$, and the space can
be equipped with product measures $\dd\mu_i\times \dd\nu_i$, $i\le N$,
that are not necessarily the same, such that each $(\omega_i, Y_i)$
has a joint density with respect to $\dd\mu_i\times\dd\nu_i$ as 
\begin{align} \label{eq:hidden-joint}
  \prob\Cbr{\omega_i\in \dd z, Y_i\in \dd y}
  = \frac{g_i(x_i(z)\tp\vtheta)k_i(z,y)\mu_i(\dd z)
    \nu_i(\dd y)}{Z_i(\vtheta)},
\end{align}
where $k_i$, $g_i$ and $x_i$ are known functions with $x_i:\Omega\to
\Reals^p$, $\vtheta\in\Reals^p$ is the
true parameter value which is unknown, and for each $\vf
u\in\Reals^p$, $Z_i(\vf u)$ is the normalizing constant
$$
Z_i(\vf u) = \int g_i(x_i(z)\tp\vf u) k_i(z,y)
\mu_i(\dd z) \nu_i(\dd y).
$$

Suppose that only $\eno Y N$ are observed, while $\eno\omega N$ are
hidden.  The (log)-likelihood function is then
$$
\ell(\vf u) = \ell(\vf u, \eno Y N)
=-\sm i N \ln \int g_i(x_i(z)\tp\vf u) k_i(z, Y_i)
\mu_i(\dd z) + \sm i N \ln Z_i(\vf u).
$$
We next consider the local stochastic Lipschitz continuity of
$\ell(\vf u)$ at the true parameter value $\vtheta$.  By
\begin{gather*}
  \frac{\displaystyle\int g_i(x_i(z)\tp\vf u) k_i(z, Y_i)
    \mu_i(\dd z)}
  {\displaystyle\int g_i(x_i(z)\tp\vtheta) k_i(z, Y_i)
    \mu_i(\dd z)} 
  = \mean\Sbr{\frac{g_i(x_i(\omega_i)\tp\vf u)}{
      g_i(x_i(\omega_i)\tp\vtheta)}\ \vline\ Y_i},
\end{gather*}
and
$$
\frac{Z_i(\vf u)}{Z_i(\vtheta)} = \mean\Sbr{
  \frac{g_i(x_i(\omega_i)\tp\vf u)}{g_i(x_i(\omega_i)\tp\vtheta)}},
$$
we have
\begin{align*}
  \ell(\vf u) - \ell(\vtheta)
  &
  =-\sm i N \ln 
  \frac{\displaystyle
    \int g_i(x_i(z)\tp\vf u) k_i(z, Y_i) \mu_i(\dd z)
  }{\displaystyle
    \int g_i(x_i(z)\tp\vtheta) k_i(z, Y_i) \mu_i(\dd z)
  } 
  +\sm i N \ln \frac{Z_i(\vf u)}{Z_i(\vtheta)}\\
  &
  =
  -\sm i N \ln \mean\Sbr{
    \frac{g_i(x_i(\omega_i)\tp\vf u)}{g_i(x_i(\omega_i)\tp\vtheta)}
    \ \vline\ Y_i
  }
  +
  \sm i N\ln \mean\Sbr{
    \frac{g_i(x_i(\omega_i)\tp\vf u)}{g_i(x_i(\omega_i)\tp\vtheta)}
  }.
\end{align*}

Let $D$ be the search domain and suppose it is known that $\vtheta\in
D$.  For the tail of
$$
\sup_{\vf u\in D\setminus\{\vtheta\}} 
\frac{|\dev{-.35ex}{\ell(\vf u) - \ell(\vtheta)}|}
{\normx{\vf u-\vtheta}_1},
$$
our analysis is based on the following assumption.
\begin{assumption} \label{a:hidden}
  There is $M_X>0$, such that w.p.~1,
  $$
  \normx{x_i(\omega_i)}_\infty \le M_X, \quad \text{all}\ \  i\le N.
  $$
  For all $i\le N$, $g_i(t)$ is first order differentiable.  Moreover,
  there are $0<A_g<B_g<\infty$, $F_1<\infty$, and $F_2<\infty$, such
  that
  $$
  A_g \le g_i(t)\le B_g, \quad
  |g_i'(t)| \le F_1, \quad |g_i'(t) - g_i'(s)|\le F_2|t-s|, \quad
  \text{all}\ \ i\le N.
  $$
\end{assumption}

We need to introduce some constants.  Denote
$$
R_D = \sup_{\vf u\in D} \normx{\vf u-\vtheta}_1, \quad
I_g=\Sbr{A_g/B_g, B_g/A_g}.
$$
Define for $z\in \Reals$ 
$$
\varrho(z) = \begin{cases}
  z^{-1} \ln(1+z) - 1 & z\not=0 \\
  0 & z=0
\end{cases}
$$
It is easy to see that $\varrho$ is smooth and strictly decreasing on
$(-1,\infty)$ with $\varrho(0)=0$.  Denote
$$
\varrho_0:=\sup_{t\in I_g} |\varrho(t-1)|<\infty, \quad
\varrho_1:= \sup_{t\in I_g} |\varrho'(t-1)|<\infty.
$$
Denote the following constants
\begin{gather*}
  \psi_1= F_1/A_g, \quad
  \psi_2= F_2 M_X/(2 A_g), \quad
  \psi_3= \min\Grp{2F_1, F_2 M_X R_D/2}/A_g, \\
  \psi_4= [\psi_1 \varrho_0 + \psi_3 (1+\varrho_0)]M_X, \quad
  \psi_5= 2\psi_1 M_X \varrho_1, \quad
  \psi_6= 2(\varrho_0+\psi_3 M_X\varrho_1).
\end{gather*}

\begin{theorem} \label{thm:hidden}
  Denote
  $$
  S_X = \mx j p \sqrt{\sm i N x_{ij}(\omega_i)^2}
  $$
  where $x_{ij}(\omega_i)$ is the $j$th coordinate of
  $x_i(\omega_i)$. 
  Under Assumption \ref{a:hidden}, for any $q_0$, $q_1\in (0,1)$ with
  $q_0+q_1<1$, w.p.~at least $1-q_0-q_1$,
  \begin{align*}
    \sup_{\vf u\in D\setminus\{\vtheta\}} 
    \frac{|\dev{-.35ex}{\ell(\vf u) - \ell(\vtheta)}|}
    {\normx{\vf u-\vtheta}_1} 
    &\le 
    2\sqrt{2}R_D\mean S_X \Grp{A\sqrt{\ln(2p)} +
      2\psi_3(\psi_5+\psi_6)
    } \\
    &\qquad
    + 
    \sqrt{2N}\Grp{\psi_1 M_X \sqrt{\ln(2p/q_0)}+
      2\psi_4\sqrt{\ln(p/q_1)}},
  \end{align*}
  where $A= 2\psi_2 (1+\psi_6) + (\psi_1 + 2
  \psi_2 R_D + 2\psi_3)(\psi_5+\psi_6)$.
\end{theorem}

To see how Theorem \ref{thm:hidden} may be used, consider the following
Lasso type estimator
\begin{align} \label{eq:lasso-hidden}
  \est\vtheta = \argf\min_{\vf u\in D}
  \Cbr{
    \ell(\vf u) + \lambda\normx{\vf u}_1
  },
\end{align}
where $\lambda>0$ is a tuning parameter.   The next result is in the
same spirit as Theorem \ref{thm:lasso} and actually simpler, as no
design matrices are involved.  Furthermore, it holds in a more general
setting than the hidden variable case.
\begin{prop} \label{prop:hidden}
  Let $\ell(\vf u)$ be a stochastic process indexed by $\vf u\in
  D\subset \Reals^p$.  Fix $\vtheta\in D$.  Let
  $S:=|\sppt(\vtheta)|\le p/2$ and $q\in (0,1)$.  Suppose the
  following two conditions are satisfied.
  \begin{itemize}
  \item[1)] There is a constant $C_\ell>0$, such that $\mean \ell(\vf
    u) - \mean\ell(\vtheta)\ge C_\ell\normx{\vf u-\vtheta}_2^2$.
  \item[2)] There is $M_q>0$, such that 
    $$
    \prob\Cbr{
      \sup_{\vf u\in D\setminus\{\vtheta\}} 
      \frac{|\dev{-.35ex}{\ell(\vf u) - \ell(\vtheta)}|}
      {\normx{\vf u-\vtheta}_1}
      \ge M_q
    } \le q.
    $$
  \end{itemize}
  Given $K>1$, let $\lambda = K M_q$ in \eqref{eq:lasso-hidden}.  Then
  w.p.~at least $1-q$,
  $$
  \normx{\est\vtheta-\vtheta}_2 \le
  \sqrt{2+\frac{(K+1)^2}{(K-1)^2}}  \frac{(K+1) M_q
    \sqrt{S}}{C_\ell}.
  $$
\end{prop}

Proposition \ref{prop:hidden} requires two conditions.  On the one
hand, Theorem \ref{thm:hidden} can be used to derive Condition 2).
On the other, Condition 1) requires extra assumptions to establish.
In the context of hidden variables, since
\begin{align} \label{eq:hidden-mean}
  \mean\Sbr{\ell(\vf u) - \ell(\vtheta)}
  =
  -\sm i N \mean\Sbr{
    \ln \mean\Sbr{
      \frac{g_i(x_i(\omega_i)\tp\vf u)}{g_i(x_i(\omega_i)\tp\vtheta)}
      \ \vline\ Y_i
    }
  }
  +
  \sm i N \ln \mean\Sbr{
    \frac{g_i(x_i(\omega_i)\tp\vf u)}{g_i(x_i(\omega_i)\tp\vtheta)}
  },
\end{align}
we need some assumptions on the structure of $(\omega_i, Y_i)$.  We
next consider a case in which both $\omega_i$ and $Y_i$ are
processes.  One could have a hidden Markov model in mind, with
$\omega_i$ the hidden Markov process and $Y_i$ the observations.

Suppose $(\omega_i, Y_i)$ are i.i.d.\ and for each $i\le N$,
$\omega_i = (\omega_{i1}, \ldots, \omega_{in})$ and $Y_i=(Y_{i1},
\ldots, Y_{in})$ are jointly distributed processes, such that all
$\omega_{it}$ take values in a common alphabet
$A=\{1,\ldots, 1+L\}$, and conditioning on $\omega_i$, $Y_{i1}$, \ldots,
$Y_{in}$ are independent with $Y_{it} \sim N(\omega_{it}, \sigma^2)$.
Suppose $\omega_i$ follows a tilted version of a 
baseline distribution $\pi_0(z)$
$$
\pi(z\gv\vtheta) \propto \pi_0(z) \exp\Cbr{
  \sm tn\sm aL \cf{z_i = a} \theta_{ia}
}, \quad z=(\eno z n)\in A^n.
$$
Note that while $A$ has $L+1$ different letters, to make sure the 
identifiability of $\theta_{ia}$, only $L$ parameters are associated
with each $t\le n$.

Let $\mu_i = \pi_0$ and $\nu_i$ the Lebesgue measure on $\Reals$.  For
$i\le N$, let 
$k_i(z,y) = \prod_{t\le n} f((z_t-y_t)/\sigma)$ with $f$ the density
of $N(0,1)$, $x_i(z)\in \{0,1\}^{nL}$ with
the $((t-1)L+a)$-th entry equal to $\cf{z_t=a}$, and
$\vtheta=(\theta_{11}, \ldots, \theta_{1L}, \ldots, \theta_{n1},
\ldots, \theta_{nL})$.  Finally, let $g_i(x) = e^x$.  Then the above
model can be formulated as in \eqref{eq:hidden-joint}.  Denote $X_i =
x_i(\omega_i)$.  By
\eqref{eq:hidden-mean}, 
\begin{align*}
  \mean\Sbr{\ell(\vf u) - \ell(\vtheta)}
  =
  N\Sbr{
    \ln \mean e^{X_1\tp(\vf u-\vtheta)}-
    \mean\ln \mean(e^{X_1\tp(\vf u-\vtheta)}\ \vline\ Y_1)
  }.
\end{align*}

Suppose that it is known that $\vtheta\in D$, where $D\in \Reals^{nL}$
is a bounded set.  As discussed earlier, the concern here is Condition
1) in Proposition \ref{prop:hidden}.  We can make the following
assertion.
\begin{prop} \label{prop:hmm}
  Suppose $\pi_0(z)>0$ for all $z\in A^n$.  Then Condition 1) of
  Theorem \ref{thm:hidden} is satisfied.
\end{prop}

\subsection{Proof of Theorem \ref{thm:hidden}}
For notational easy, denote
\begin{align*}
  \vf v = \vf u-\vtheta, \quad X_i = x_i(\omega_i), \quad
  \mean_i(\cdot) = \mean(\cdot\gv Y_i)
\end{align*}
and
\begin{align*}
  \gamma_i(\vf v) = \mean_i\Sbr{
    \frac{g_i(X_i\tp(\vtheta+\vf v))}{g_i(X_i\tp\vtheta)}
  } -1, \quad
  \gamma(\vf v) = \mean\Sbr{
    \frac{g_i(X_i\tp(\vtheta+\vf v))}{g_i(X_i\tp\vtheta)}
  } -1.
\end{align*}
Note that, because $\vtheta$ is fixed even though unknown,
$\gamma_i(\vf v)$ is a random function only dependent on $\vf v$ and
$Y_i$, while $h(\vf v)$ is a nonrandom function only dependent on $\vf
v$.  Then 
\begin{align} \label{eq:hidden-lik}
  \ell(\vf u) - \ell(\vtheta)
  = -\sm i N \gamma_i(\vf v) (\varrho(\gamma_i(\vf v))+1)
  +
  \sm i N \gamma(\vf v) (\varrho(\gamma(\vf v))+1).
\end{align}
Define functions
\begin{align} \label{eq:hidden-def3}
  \lambda_i(s) = (\ln g_i)'(s) = \frac{g_i'(s)}{g_i(s)}, \quad
  \varphi_i(s,t)
  =
  \begin{cases}
    \displaystyle
    \frac{g_i(s+t)-g_i(s)}{
      g_i(s)t} - \lambda_i(s)
    & t\not=0 \\[2ex]
    0 & t=0,
  \end{cases}
\end{align}
of $s$, $t\in\Reals$.  Define $\vf z_i = (z_{i1}, \ldots, z_{ip})$ and
$s_i(\vf v) = (s_{i1}(\vf v), \ldots, s_{ip}(\vf v))$ with
\begin{align} \label{eq:hidden-s-decomp}
  z_{ih} = \mean_i\Sbr{\lambda_i(X_i\tp\vtheta) X_{ih}},
  \quad
  s_{ih}(\vf v) = \mean_i\Sbr{\varphi_i(X_i\tp\vtheta,X_i\tp\vf v)
    X_{ih}}, \quad i\le N, \ h\le p.
\end{align}
Note that each $z_{ih}$ is a function only in $Y_i$, and each
$s_{ih}(\vf v)$ is a function only in $\vf v$ and $Y_i$.  Then
\begin{align} \label{eq:hidden-h-decomp}
  \gamma_i(\vf v)
  = \mean_i\Sbr{
    \lambda_i(X_i\tp\vtheta) X_i +
    \varphi_i(X_i\tp\vtheta,X_i\tp\vf v) X_i
  }\tp \vf v = 
  (\vf z_i + s_i(\vf v))\tp \vf v
\end{align}
which combined with \eqref{eq:hidden-lik} yields
\begin{align*}
  \dev{-.35ex}{\ell(\vf u) - \ell(\vtheta)}
  = -\sm iN \dev{-.45ex}{(1+ \varrho(\gamma_i(\vf v)))(\vf z_i
    + s_i(\vf v))}\tp\vf v.
\end{align*}
For $i\le N$ and $h\le p$, write $\zeta_{ih}(\vf v)= s_{ih}(\vf v) +
z_{ih} \varrho(\gamma_i(\vf v)) + s_{ih} (\vf v)\varrho(\gamma_i(\vf
v))$ so that the above equation can be written as
\begin{align*}
  \dev{-.35ex}{\ell(\vf u) - \ell(\vtheta)}
  = -\sm h p \Grp{\sm iN \dev{-.45ex}{z_{ih}}} v_h
  - \sm h p \Grp{\sm iN\dev{-.45ex}{\zeta_{ih}(\vf v)}} v_h.
\end{align*}
Define for $h\le p$,
$$
W_h = \sup_{\vf u\in D} \Abs{\sm i N \dev{-.35ex}{\zeta_{ih}(\vf v)}}.
$$
Then
\begin{align}
  \sup_{\vf u\in D\setminus\{\vtheta\}}
  \frac{|\dev{-.35ex}{\ell(\vf u) - \ell(\vtheta)}|}
  {\normx{\vf u-\vtheta}_1}
  \le \mx hp \Abs{\sm iN \dev{-.35ex}{z_{ih}}} + \mx hp W_h.
  \label{eq:hidden-lip}
\end{align}

\begin{lemma} \label{lemma:hidden}
  (1) W.p.~1, the following inequalities hold simultaneously,
  \begin{gather*}
    \Abs{\lambda_i(X_i\tp\vtheta)}
    \le \psi_1, \qquad
    \frac{g_i(X_i\tp\vf u)}{g_i(X_i\tp\vtheta)} \in  I_g.
  \end{gather*}
  (2) $\varphi_i(s,0)\equiv 0$ and w.p.~1, for all $s\in\Reals$, $i\le
  N$ and $h\le p$, $\varphi_i(s, \cdot) X_{ih}$ is $\psi_2$-Lipschitz
  and $\Abs{\varphi_i(s, X_i\tp(\vf u-\vtheta))}\le \psi_3$ for
  all $\vf u\in D$.
\end{lemma}

Given $h\le p$, from Lemma \ref{lemma:hidden}, w.p.~1, for all $i\le
N$, $|z_{ih}|\le \psi_1 M_X$.  Then by union-sum inequality and Hoeffding
inequality,
$$
\prob\Cbr{\mx h p\Abs{\sm i N \dev{-.35ex}{z_{ih}}}
  \ge \sqrt{N} \psi_1 M_X t}\le
\sm h p \prob\Cbr{\Abs{\sm i N \dev{-.35ex}{z_{ih}}} \ge \sqrt{N}
  \psi_1 M_X t} \le 2 p e^{-t^2/2}.
$$
Letting $t=\sqrt{2 \ln(2p/q_0)}$ then yields the following bound on
the first term in \eqref{eq:hidden-lip}
\begin{align} \label{eq:hidden-alpha}
  \prob\Cbr{\mx h p\Abs{\sm i N \dev{-.35ex}{z_{ih}}}
    \ge \psi_1 M_X \sqrt{2N\ln(2p/q_0)}}\le q_0.
\end{align}

Given $h$, from Lemma \ref{lemma:hidden}, w.p.~1, for all $i\le N$ and
$\vf v = \vf u-\vtheta$ with $\vf u\in D$, $|\zeta_{ih}(\vf v)|\le
\psi_4$, so $-\psi_4-\mean\zeta_{ih}(\vf v) \le
\dev{-.35ex}{\zeta_{ih}(\vf v)}\le \psi_4 -\mean\zeta_{ih}(\vf
v)$.  Then by Lemma \ref{lemma:mc}, inequality \eqref{eq:f-hoeff},
$$
\Pr\Cbr{W_h > \mean W_h + 2\psi_4\sqrt{2Ns}} \le
e^{-s}.
$$
By union-sum inequality, it follows that
\begin{align} \label{eq:hidden-xi}
  \Pr\Cbr{\mx h p W_h > \mx h p \mean W_h + 2\psi_4 \sqrt{2N\ln
      (p/q_1)}} \le  q_1.
\end{align}
  
We need to bound $\mean W_h$ for each $h\le p$.  Since $\xi_{ih}(\vf
v)$ are continuous in $\vf v$ and bounded, by dominated convergence
argument, we can apply symmetrization (\cite{ledoux:91}, Lemma 6.3) to
get
\begin{align} 
  \frac{\mean W_h}{2}
  &\le
  \mean \sup_{\vf u\in D} \Abs{\sm i N \rx_i\zeta_{ih}(\vf v)}
  \le
  \mean \sup_{\vf u\in D} \Abs{\sm i N \rx_i s_{ih}(\vf v)} 
  + L_h\Sp 1 + L_h\Sp 2,
  \label{eq:hidden-symm}
\end{align}
where $\eno\rx N$ are i.i.d.\ Rademacher variables independent of
$(\omega_i, Y_i)$, and
\begin{align*}
  L_h\Sp 1 = \mean \sup_{\vf u\in D} \Abs{\sm i N \rx_i     
    z_{ih} \varrho(\gamma_i(\vf v))
  },
  \quad
  L_h\Sp 2 = \mean \sup_{\vf u\in D}
  \Abs{\sm i N \rx_i  s_{ih}(\vf v)\varrho(\gamma_i(\vf v))},
\end{align*}
Define
\begin{align*}
  U_h = \sup_{\vf u\in D}
  \Abs{\sm i N \rx_i \varphi_i(X_i\tp\vtheta,X_i\tp\vf v)
    X_{ih}
  }, \quad
  Q = \sup_{\vf u\in D} \Abs{\sm i N \rx_i \gamma_i(\vf u-\vtheta)}.
\end{align*}
To continue, we need the following result.
\def\EUh{2\sqrt{2\ln (2p)} \psi_2 R_D \mean S_X}
\def\EQ{R_D(\sqrt{2\ln(2p)} \psi_1\mean S_X + M_U)}
\begin{lemma} \label{lemma:Uh-Q}
  For each $h\le p$,
  \begin{align} \label{eq:s-Uh}
    \mean \sup_{\vf u\in D} 
    \Abs{
      \sm i N \rx_i s_{ih}(\vf v)
    } \le \mean U_h \le \EUh.
  \end{align}
  Furthermore, let
  $$
  M_U=
  2\sqrt{2} \Sbr{\sqrt{\ln(2p)} \psi_2 R_D 
    +\psi_3 (\sqrt{\ln p}+1)}
  \mean S_X.
  $$
  Then
  $$
  \mean\mx h p U_h \le  M_U, \quad
  \mean Q \le \EQ.
  $$
\end{lemma}

To bound $L_h\Sp 1$, by Fubini theorem,
$$
L_h\Sp 1 = \mean_Y \mean_\rx \sup_{\vf u\in D} \Abs{\sm i N \rx_i
  z_{ih}\varrho(\gamma_i(\vf v))
}
$$
Given $\eno Y N$, $z_{ih}$ are fixed and $\gamma_i(\vf v)$ become
nonrandom function in $\vf v=\vf u-\vtheta$.  By Lemma
\ref{lemma:hidden}, the nonrandom function $t \to
z_{ih}\varrho(t)$ maps 0 to 0 and is  $(\psi_5/2)$-Lipschitz.
Then by Lemma \ref{lemma:u-compare},
\begin{align}
  L_h\Sp 1 \le \psi_5 \mean_Y \mean_\rx \sup_{\vf u\in D}
  \Abs{\sm iN \rx_i \gamma_i(\vf v)} = \psi_5 \mean Q.
  \label{eq:L-h1}
\end{align}

To bound $L_h\Sp 2=\sup_{\vf u\in D} \Abs{\sm i N \rx_i s_{ih}(\vf
  v)\varrho(\gamma_i(\vf v))}$, we have to use the multivariate
comparison results in Section \ref{sec:compare}.  Given $\eno Y N$,
both $s_{ih}(\vf v)$ and $\gamma_i(\vf v)$ with $\vf v = \vf
u-\vtheta$ are nonrandom functions of $\vf u\in D$.  Let $g(s,t) = s
\varrho(t)$ for $i\le N$ and
$$
T = \Cbr{\vf t=(\eno{\vf t} N): \vf t_i = (s_{ih}(\vf v), \gamma_i(\vf
  v)),\ i\le N,\ \vf v=\vf u-\vtheta, \ \vf u\in D}.
$$
Then 
$$
\mean_\rx \sup_{\vf u\in D} \Abs{\sm i N \rx_i
  s_{ih}(\vf v)\varrho(\gamma_i(\vf v))
}
= \mean_\rx \sup_{\vf t\in T} \Abs{\sm i N \rx_i g(\vf t_i)}.
$$

By Lemma \ref{lemma:hidden}, w.p.~1, for all $i\le N$, $h\le p$, and
$\vf u\in D$, $(s_{ih}(\vf v), \gamma_i(\vf v))\in J$, where
$$
J=\Cbr{(s,t-1): |s|\le \psi_3 M_X,\ t\in I_g}.
$$

\begin{lemma} \label{lemma:bivariate-lip}
  $g(s,t)$ is $(\psi_6/2, \ell_\infty)$-Lipschitz on $J$.
  Furthermore, $|g(s,t)|\le (\psi_6/2) \min(|s|, |t|)$.
\end{lemma}

From Lemma \ref{lemma:bivariate-lip} and Theorem \ref{thm:compare0},
\begin{align*}
  \mean_\rx \sup_{\vf t\in T} \Abs{\sm i N \rx_i g(\vf t_i)}
  \le \psi_6\Grp{
    \mean_\rx \sup_{\vf t\in T} \Abs{\sm i N \rx_i t_{i1}}
    +
    \mean_\rx \sup_{\vf t\in T} \Abs{\sm i N \rx_i t_{i2}}
  }.
\end{align*}
Integrating over $\eno Y N$, we thus get
\begin{align} \label{eq:L-h2}
  L_h\Sp 2
  &
  \le \psi_6\Grp{
    \mean \sup_{\vf u\in D} \Abs{\sm i N \rx_i \gamma_i(\vf v)}
    +
    \mean \sup_{\vf u\in D} \Abs{\sm i N \rx_i s_{ih}(\vf v)}
  }
  \le \psi_6\mean (Q+U_h),
\end{align}
where the second inequality is due to \eqref{eq:s-Uh}.  Combine
\eqref{eq:hidden-symm}, \eqref{eq:L-h1}, \eqref{eq:L-h2} and
Lemmas \ref{lemma:Uh-Q},
\begin{align*}
  \frac{\mean W_h}{2}
  \le\
  &
  (1+\psi_6)\mean U_h + (\psi_5+\psi_6)\mean Q \\
  \le\
  &
  (1+\psi_6) \EUh + 
  (\psi_5+\psi_6)R_D(\sqrt{2\ln(2p)} \psi_1 \mean S_X
  + M_U) 
\end{align*}
Note that the bound holds for all $h\le p$.  Incorporate the bound
into \eqref{eq:hidden-xi}.  Together with \eqref{eq:hidden-lip} and
\eqref{eq:hidden-alpha}, this finishes the proof.

\subsection{Proof of Lemma \ref{lemma:Uh-Q}}

To prove the Lemma, we need the following result.
\begin{lemma} \label{lemma:mean-max}
  Suppose $\xi\ge 0$ such that for some $a,b,c\ge 0$ and $d\ge 1$,
  \begin{align} \label{eq:xi-tail}
    \prob\Cbr{\xi > a + b\sqrt{s} + c s} \le d e^{-s}, \quad s>0.
  \end{align}
  Then $\mean \xi \le a+b(\sqrt{\ln d}+1) + c(\ln d+1)$.
\end{lemma}
\begin{proof}
  First, if $b=c=0$, then $\prob\Cbr{\xi > a}<d e^{-s}$ for any $s>0$.
  Let $s\to\infty$ to get $\xi\le a$ and hence $\mean \xi\le
  a + b(\sqrt{\ln d}+1) + c(\ln d+1)$.  Assume $b+c>0$.  Condition
  \eqref{eq:xi-tail} implies
  $$
  \prob\Cbr{\xi > a + b\sqrt{s+\ln d} + c(s+\ln d)}
  \le e^{-s}, \quad s>0.
  $$
  By $\sqrt{s+\ln d} \le \sqrt{s}+\sqrt{\ln d}$,
  $$
  \prob\Cbr{\xi > a_0 + f(s)} \le e^{-s}, \quad s>0,
  $$
  where $a_0 = a + b\sqrt{\ln d} + c\ln d$ and $f(s)=b\sqrt{s} + c s$
  is a 1-to-1 and onto mapping $[0,\infty)\to [0,\infty)$.  Let
  $f^{-1}$ be the inverse of $f$.  Then
  \begin{align*}
    \mean \xi
    &
    = \int_0^{a_0} \prob\Cbr{\xi>t}\,\dd t + 
    \int_0^\infty \prob\Cbr{\xi>a_0+t}\,\dd t 
    \le a_0 + 
    \int_0^\infty \exp\Cbr{-f^{-1}(t)}\,\dd t \\
    &= a_0 + \int_0^\infty e^{-s} f'(s)\,\dd s
    = a_0 + \int_0^\infty e^{-s} (bs^{-1/2}/2 + c)\,\dd s
    \le a_0 + b+ c,
  \end{align*}
  which completes the proof.
\end{proof}

\begin{proof}[Proof of Lemma \ref{lemma:Uh-Q}]
  By independence and Jensen inequality,
  \begin{align*}
    \sup_{\vf u\in D} \Abs{\sm i N \rx_i s_{ih}(\vf v)}
    &
    =
    \sup_{\vf u\in D} 
    \Abs{\mean\Sbr{\sm i N \rx_i \varphi_i(X_i\tp\vtheta, X_i\tp\vf v)
      X_{ih}\ \vline\ \rx_i, Y_i,\, i\le N}} \\
    &
    \le 
    \mean\Sbr{
      U_h \gv \rx_i, Y_i,\, i\le N
    }.
  \end{align*}
  Integrating over $\eno Y N$ and $\eno\rx N$ leads to the first
  inequality in \eqref{eq:s-Uh}.  Given $\eno XN$, by Lemma
  \ref{lemma:hidden}, $t \to
  \varphi_i(X_i\tp\vtheta, t) X_{ih}$ maps 0 to 0 and is
  $\psi_2$-Lipschitz.  Therefore, by Lemma \ref{lemma:u-compare},
  H\"older inequality, and Lemma \ref{lemma:massart},
  \begin{align*} 
    \mean_\rx U_h 
    &\le
    2 \psi_2 \mean_\rx \sup_{\vf u\in D}\Abs{\sm i N \rx_i X_i\tp \vf v}
    \le 2\psi_2 R_D \mean_\rx \mx j p\Abs{\sm i N \rx_i X_{ij}}
    \nonumber \\
    &
    \le 2\psi_2 R_D \mx j p \sqrt{\sm i N X_{ij}^2}\times
    \sqrt{2\ln(2p)}
    = 2\sqrt{2\ln(2p)}\psi_2 R_D S_X.
  \end{align*}
  Integrating over $\eno X N$ yields the second inequality in
  \eqref{eq:s-Uh}.

  On the other hand, given $\eno X N$, by Lemma
  \ref{lemma:hidden},  $|\varphi_i(X_i\tp\vtheta, X_i\tp\vf v)|
  \le \psi_3$ for each $\vf u\in D$.  Since $\eno\rx N$ are i.i.d.\
  Rademacher variables independent of $\eno X N$, then by Lemma
  \ref{lemma:mc} inequality \eqref{eq:f-hoeff}, for each $h\le p$,
  \begin{align} \label{eq:Uh-tail}
    \prob\Cbr{U_h \ge \mean_\rx U_h + 2\psi_3\sqrt{2 s\sm iN
        X_{ih}^2} \ \vline\ \eno X N} \le e^{-s}.
  \end{align}
  
  Incorporate the bound on $\mean_\rx U_h$ into \eqref{eq:Uh-tail} and 
  apply union-sum inequality to 
  \begin{align*}
    \prob\Cbr{\mx hp U_h \ge 
      2\sqrt{2\ln(2p)} \psi_2 R_D S_X      
      + 2\sqrt{2s}\psi_3 S_X
      \ \vline\ \eno X N
    }\le pe^{-s}.
  \end{align*}
  By Lemma \ref{lemma:mean-max}, we get 
  \begin{align*}
    \mean\Sbr{\mx hp U_h \gv\eno X N}
    \le
    2\sqrt{2\ln(2p)} \psi_2 R_D S_X
    +2\sqrt{2}\psi_3 S_X (\sqrt{\ln p}+1)
  \end{align*}
  Integrating over $\eno X N$ yields the bound on $\mean\mx h p U_h$.

  Finally, by \eqref{eq:hidden-h-decomp} and $\normx{\vf u -
    \vtheta}_1\le R_D$ for all $\vf u\in D$,
  \begin{align*}
    Q = \sup_{\vf u\in D} 
    \Abs{\sm hp\sm i N \rx_i (z_{ih} + s_{ih}(\vf v)) v_h} 
    \le R_D\sup_{\vf u\in D} 
    \mx h p \Abs{\sm i N \rx_i(z_{ih} + s_{ih}(\vf v))
    }.
  \end{align*}
  Following the proof for the first inequality in \eqref{eq:s-Uh},
  \begin{align*}
    \mean Q 
    \le R_D 
    \Grp{
      \mean\mx h p \Abs{\sm i N \rx_i\lambda_i(X_i\tp\vtheta) X_{ih}}
      + \mean\mx h p U_h
    }.
  \end{align*}
  Given $\eno X N$, by Lemma \ref{lemma:hidden}, $\sm iN
  \Grp{\lambda_i(X_i\tp\vtheta) X_{ih}}^2 \le \psi_1^2 S_X^2$ for all
  $h\le p$.   Therefore, by Lemma \ref{lemma:massart} and Fubini
  theorem,
  \begin{align*}
    \mean \mx h p \Abs{\sm i N \rx_i\lambda_i(X_i\tp\vtheta) X_{ih}}
    \le \sqrt{2\ln (2p)}\,\psi_1\mean S_X,
  \end{align*}
  which together with the bound on $\mean\mx h p U_h$ yields the
  desired bound on $\mean Q$.
\end{proof}

\bibliographystyle{acmtrans-ims2}

\section*{Appendix: miscellaneous proofs}
In this section, we collect proofs for the lemmas and propositions in
the main text.
\begin{proof}[Proof of Lemma \ref{lemma:mc}]
  First, assume $D$ is finite.  Let $D_+=D\times \{1\}$ and
  $D_-=D\times \{-1\}$ and denote $T=D_+\cup D_-$.   The random
  vectors $X_i = (\sigma f_i(\vf u), (\vf u,\sigma)\in T)$, $i\le N$,
  are independent taking values in $\Reals^T$.  It is easy to check
  that
  $$
  W = \max_{\vf u\in D} \Cbr{\sm i N f_i(\vf u),
    -\sm i N f_i(\vf u)} = \max_{t\in T} \sm i N X_{i,t},
  $$
  where $X_{i,t}$ denotes the $t$-th coordinate of $X_i$.  Since 
  $a_i\le X_{i,t}\le b_i$ for $t=(\vf u, 1)$ and $-b_i \le X_{i,t} \le
  -a_i$ for $t=(\vf u, -1)$, by Theorem 9 of \cite{massart:00:ap},
  letting $L^2 = \sm i N (b_i -a_i)^2$,
  $$
  \prob\Cbr{W \ge \mean W + x} \le \exp\Cbr{-\frac{x^2}{2 L^2}},
  \quad x>0,
  $$
  which implies \eqref{eq:f-hoeff}.

  Still assuming $D$ is finite, assume moreover that $\mean f_i(\vf
  u)=0$ for all $i\le N$ and $\vf u\in D$.  For each $t=(\vf
  u,\sigma)\in T$, define $s_t = (s\sp 1_t, \ldots, s\sp N_t)$, with
  each $s\sp i_t$ being the map $x\to x_t/M$.   Then w.p.~1, $s\sp
  i_t(X_i) = \sigma f_i(\vf u)/M \in [-1,1]$ with mean 0 for $i\le N$
  and
  \begin{align*}
    \sup_{t\in T} \var\Grp{\sm i N s\sp i_t(X_i)}
    &= \sup_{\vf u\in D,\,\sigma=\pm 1} 
    \var\Grp{\sm i N \sigma f_i(\vf u)/M}  \\
    &
    = \sup_{\vf u\in D} 
    \sm i N \var\Grp{f_i(\vf u)/M} \le (S/M)^2.
  \end{align*}
  We next apply Theorem 1.1 of \cite{klein:05:ap} to $\tilde W =
  \sup_{t\in T} \sm i N s\sp i_t(X_i)$.  Let $w = 2 \mean\tilde W +
  (S/M)^2$.  Then for any $a>0$,
  $$
  \prob\Cbr{\tilde W \ge \mean\tilde W + a}
  \le \exp\Cbr{-\frac{a^2}{2 w + 3a}}.
  $$
  For any $s>0$, $a=(3s + \sqrt{9s^2 + 8ws})/2$ is the unique positive
  solution to $a^2/(2w+3a)=s$.  By using $\sqrt{x+y}\le \sqrt{x} +
  \sqrt{y}$ and $2\sqrt{xy}\le x + y$ for $x,y\ge 0$, it is seen that
  $a\le \mean\tilde W + (S/M)\sqrt{2s} + 4s$.  So
  $$
  \prob\Cbr{\tilde W \ge 2 \mean\tilde W + (S/M)\sqrt{2s} + 4s}
  \le e^{-s}.
  $$
  Since $\tilde W = W/M$, \eqref{eq:mc} then follows.

  For an arbitrary $D$, by the path continuity of $f_i$, $W =
  \lim_n \sup_{\vf u \in D_n} |\sm i N f_i(\vf u)|$, with
  $D_1 \subset D_2 \ldots$ being a (nonrandom) sequence of
  finite subsets of $D$.  Then the proof is complete by monotone
  convergence.
\end{proof}

\begin{proof}[Proof of Lemma \ref{lemma:phi}]
  Denote $f_i(\cdot) = \gamma_i(\cdot,Y_i)$.  For any $\vf
  s\in\Reals^k$, whether or not $s_j=0$,
  $$
  \varphi_{ij}(\vf s) = \int_0^1 \Grp{
    \partial_j f_i(\vf c_i+\bar\pi_{j-1} \vf s + s_j u \vf e_j)
    -     \partial_j f_i(\vf c_i)}\,\dd u,
  $$
  where $\vf e_j$ is the $j$th standard basis vector of $\Reals^k$.
  It follows that $\varphi_{ij}(\vf 0)=0$ and for $\vf s\in\Reals^k$,
  $$
  |\varphi_{ij}(\vf s)|
  \le \int_0^1
  \Abs{\partial_j f_i(\vf c_i+\bar\pi_{j-1} \vf s + s_j u\vf e_j)-
    \partial_j f_i(\vf c_i)
  }\,\dd u \le 2F_1.
  $$
  Furthermore, for $\vf t\in\Reals^k$,
  $$
  |\varphi_{ij}(\vf s) - \varphi_{ij}(\vf t)|
  \le \int_0^1
    \Abs{\partial_j f_i(\vf c_i+\bar\pi_{j-1} \vf s + s_j u \vf e_j)-
      \partial_j f_i(\vf c_i+\bar\pi_{j-1} \vf t + t_j u \vf e_j)
    }\,\dd u.
  $$
  Since $\partial_j f_i$ is $(F_2, \ell_\infty)$-Lipschitz, the
  integrated function on the right hand side is no greater than
  $F_2\normx{\bar\pi_{j-1} (\vf s - \vf t) + (s_j - t_j) u \vf
    e_j}_\infty \le F_2\normx{\vf s - \vf t}_\infty$, and so the
  integral is no greater than $F_2 \normx{\vf s-\vf t}_\infty$,
  proving $\varphi_{ij}$ is $(F_2, \ell_\infty)$-Lipschitz.  In 
  particular, letting $\vf s=\vf0$ and $\vf t = Z_i(\vf u -
  \vtheta)$ gives $|\varphi_{ij}(\vf t)|\le F_2 \normx{Z_i(\vf u -
    \vtheta)}_\infty  = F_2 \max_j |Z_{ij}\tp (\vf u-\vtheta)| \le
  F_2 M_Z\normx{\vf u - \vtheta}_1 \le F_2 M_Z R_D$.  
\end{proof}

The proof of Lemma \ref{lemma:phi-glob} is similar to Lemma
\ref{lemma:phi}.

\begin{proof}[Proof of Proposition \ref{prop:hidden}]
  The proof is more or less standard (cf.\ \cite{bickel:09:as}), so we
  will be brief.  By definition of $\est\vtheta$ and the assumptions
  of Proposition \ref{prop:hidden}, w.p.~at least $1-q$,
  \begin{align*}
    C_\ell\normx{\est\vtheta-\vtheta}_2^2
    \le
    \mean\ell(\vtheta) - \mean\ell(\est\vtheta)
    &\le
    \dev{-.35ex}{\ell(\smash{\est\vtheta}) - \ell(\vtheta)} -
    \lambda(\normx{\est\vtheta}_1 - \normx{\vtheta}) \\
    &
    \le
    M_q \normx{\est\vtheta-\vtheta}_1 -  K M_q(\normx{\est\vtheta}_1 -
    \normx{\vtheta}_1).
  \end{align*}
  Let $A=\sppt(\vtheta)$ and $B$ the union of $A$ and the indices
  corresponding to the $S$ largest $|\est\theta_h|$ outside of $A$.
  Let $a = M_q/C_\ell$.  Then by the above inequality, for $J=A$, $B$, 
  $$
  \normx{\pi_J\est\vtheta - \vtheta}_2^2
  + \normx{\pi_{J^c}\est\vtheta}_2^2
  \le (K+1) a \normx{\pi_J\est\vtheta-\vtheta}_1 - (K-1)
  a\normx{\pi_{J^c} \est\vtheta}_1,
  $$
  giving
  $$
  \normx{\pi_J\est\vtheta - \vtheta}_2 \le
  (K+1) a\sqrt{|J|}, \qquad
  \normx{\pi_{J^c}\est\vtheta}_1 
  \le c\normx{\pi_J\est\vtheta - \vtheta}_1,
  $$
  with $c=(K+1)/(K-1)$.  Since $\normx{\pi_{B^c}\est\vtheta}_2^2\le
  \normx{\pi_{A^c} \est\vtheta}_1^2/S$ (cf.~\cite{candes:tao:07}), the
  second inequality in the display gives $\normx{\pi_{B^c}
    \est\vtheta}_2^2 \le 
  c^2\normx{\pi_A \est\vtheta - \vtheta}_2^2$.  The proof obtains by
  combining this and the first inequality in the display, applied to
  $A$ and $B$, respectively.
\end{proof}

\begin{proof}[Proof of Proposition \ref{prop:hmm}]
  Recall $X_1$ has $nL$ entries, such that for $t\le n$ and $a\le L$,
  the $((t-1)L+a)$-th entry is $\cf{\omega_{1t}=a}$.  Since
  $\normx{X_1}_\infty \le 1$, it is not hard to see that
  $\mean\ell(\vf u)$ is continuously differentiable in $\vf u$ and its
  Hessian at $\vtheta$ is  $H_N = N H$, where 
  $$
  H = \var(X_1) - \mean(\var(X_1\gv Y_1)) = \var(\mean(X_1\gv
  Y_1)).
  $$

  First, we show that for any $\vf u\not=\vtheta$, $\mean[\ell(\vf
  u)]>\mean[\ell(\vtheta)]$.  By Jensen inequality, 
  $\mean[\ell(\vf u)]\ge \mean[\ell(\vtheta)]$ with equality if and
  only there is a constant $c=c(\vf u)$ such that for all $y\in
  \Reals^n$, $\mean(e^{X_1\tp(\vf u-\vtheta)}\gv Y_1=y) = c$.  If for
  some $\vf u$ the equality holds, then
  $$
  \sum_{z\in A^n} e^{x(z)\tp(\vf u-\vtheta)}
  \prod_{t\le n} e^{-(y_t - z_t)^2/2\sigma^2} \pi(z\gv\vtheta) = 
  c \sum_{z\in A^n} \prod_{t\le n} e^{-(y_t - z_t)^2/2\sigma^2}
  \pi(z\gv\vtheta),
  $$
  where $x(z)\in \{0,1\}^{nL}$ such that its $((t-1)L+a)$-th entry is
  $\cf{z_t = a}$.  Denote
  \begin{align*}
    p(z) = (e^{x(z)\tp(\vf u-\vtheta)}-c) \prod_{t\le n}
    e^{-z_t^2/2\sigma^2} \pi(z\gv\vtheta)
  \end{align*}
  and make change of variable $y_t/\sigma^2\to y_t$.  Then for all $y\in
  \Reals^n$, $\sum_{z\in A^n} \exp(y\tp z) p(z)=0$.
  In other words,
  the Laplace transform of $p(z)$ is 0.  Therefore, $p(z)\equiv 0$.
  By assumption, $\pi(z\gv\vtheta)>0$ for all $z\in \Reals$.  As a
  result, $x(z)\tp(\vf u-\vtheta) = \ln c$ for all $z\in A^n$.  Let
  all $z_t=L+1$ to get $x(z)=\vf0$ and hence $\ln c=0$.  Next, given
  $t\le n$ and $a\le L$, let $z_t=a$ while for $s\not=t$, let
  $z_s=L+1$.  This yields $u_{(t-1)L+a}-\theta_{(t-1)L+a}=0$.  Thus
  $\vf u = \vtheta$.
  
  Because $\mean\ell(\vf u)$ is continuous in $\vf u\in D$, with $D$
  being bounded, to finish the proof, it remains to show that $H$ is
  positive definite.  Suppose $\vf v\tp H\vf v=0$ for
  some $\vf v$.  We need to show that $\vf v=0$.  Since $\vf v\tp H\vf
  v = \var(\mean(X_1\tp v\gv Y_1))$, there is a constant $c\in\Reals$,
  such that $\mean(X_1\tp v\gv Y_1=y)=c$ for all $y\in \Reals^n$.
  Following the argument employed to show $\mean\ell(\vf
  u)>\mean\ell(\vtheta)$ for $\vf u\not=\vtheta$, it can be shown that
  $\vf v=0$.  Thus the proof is complete.
\end{proof}

\begin{proof}[Proof of Lemma \ref{lemma:hidden}]
  Part (1) is straightforward from Assumption \ref{a:hidden}
  and the definition of $\lambda_i$ and $\varrho$.  To prove (2), observe
  \begin{align} \label{eq:varphi-int}
    \varphi_i(s,t) = \nth{g_i(s)} \int_0^1 [g_i'(s+tu)-g_i'(s)]\,\dd u.
  \end{align}
  Then
  \begin{align*}
    |\varphi_i(s,t') - \varphi_i(s,t)|
    &
    \le \nth{A_g} \int_0^1 |g_i'(s+t'u)-g_i'(s+tu)|\,\dd u \\
    &
    \le \nth{A_g} \int_0^1 F_2 |t'-t|u\,\dd u =
    \frac{F_2|t'-t|}{2A_g},
  \end{align*}
  showing $\varphi_i(s,\cdot)$ is $(F_2/2A_g)$-Lipschitz.  As a
  result, $\varphi_i(s,\cdot) X_{ih}$ is $\psi_2$-Lipschitz.  Let
  $t'=0$ and $t=X_i\tp(\vf u-\vtheta)$.  Then $\varphi_i(s,t')=0$,
  so by $|t|\le M_X R_D$, $|\varphi_i(s,t)|\le \psi_2 M_X R_D$.  On the
  other hand, \eqref{eq:varphi-int} implies  $|\varphi_i(s,t)|\le
  2F_1/A_g$.  Therefore, $|\varphi_i(s,t)|\le \psi_3$.  
\end{proof}

\begin{proof}[Proof of Lemma \ref{lemma:bivariate-lip}]
  Given $(s,t)$, $(s',t')\in J$, let $d_s = s'-s$, $d_t=t'-t$.  By
  Taylor expansion, for some $\theta\in (0,1)$, $g(s',t') - g(s,t)
  = \partial_1 g(s+\theta d_s, t+\theta d_t) d_s + \partial_2
  g(s+\theta , t+\theta d_t) d_t$.  Since $\partial_1 g(s,t) =
  \varrho(t)$ and $\partial_2 g(s,t) = s \varrho'(t)$, then by Lemma
  \ref{lemma:hidden},
  $$
  |g(s',t') - g(s,t)| \le \varrho_0 |s'-s| + \psi_3 \varrho_1
  |t'-t|.
  $$
  Therefore, $h$ is $(\psi_6/2, \ell_\infty)$-Lipschitz.  On the
  other hand, since $g(0,t)=0$, for some $\theta\in (0,1)$, 
  $$
  |g(s,t)| = |g(s,t)-g(0,t)| = |\partial_1g(\theta s,t)||s|
  \le \varrho_0|s|.
  $$
  Similarly, since $g(s,0)=0$, $|g(s,t)|\le \psi_3 \varrho_1
  |t|$.  As a result, $|g(s,t)|\le (\psi_6/2) \min(|s|,|t|)$.
\end{proof}

\end{document}